\documentclass[a4paper,dvipdfmx,11pt]{amsart}
\usepackage{amsmath}
\usepackage{amssymb}
\usepackage{enumerate}
\usepackage{amscd}
\usepackage{graphics}
\usepackage{latexsym}
\usepackage{verbatim} 
\usepackage{bbm}

\theoremstyle{plain}
\newtheorem{thm}{{\bf Theorem}}[section]
\newtheorem{cor}[thm]{{\bf  Corollary}}
\newtheorem{prop}[thm]{{\bf Proposition}}
\newtheorem{lemma}[thm]{{\bf Lemma}}
\newtheorem{fact}[thm]{{\bf Fact}}
\newtheorem{claim}[thm]{{\bf Claim}}

\theoremstyle{definition}
\newtheorem{define}[thm]{{\bf Definition}}
\newtheorem{note}[thm]{{\bf Note}}

\newtheorem{question}[thm]{{\bf Question}}

\newcommand{\cf}{\mathord{\mathrm{cf}}}

\newcommand{\dom}{\mathord{\mathrm{dom}}}
\newcommand{\size}[1]{\left\vert {#1} \right\vert}
\newcommand{\p}{\mathcal{P}}

\newcommand{\seq}[1]{\langle {#1} \rangle}

\newcommand{\ka}{\kappa}
\newcommand{\la}{\lambda}
\newcommand{\om}{\omega}

\newcommand{\bbP}{\mathbb{P}}

\newcommand{\calF}{\mathcal{F}}
\newcommand{\calG}{\mathcal{G}}

\newcommand{\calL}{\mathcal{L}}

\newcommand{\ZF}{\mathsf{ZF}}
\newcommand{\ZFC}{\mathsf{ZFC}}
\newcommand{\LT }{\mathsf{LT}}

\newcommand{\AC}{\mathsf{AC}}
\newcommand{\fix}{\mathrm{fix}}
\newcommand{\sym}{\mathrm{sym}}

\title[A note on \L o\'s's Theorem without AC]{A note on \L o\'s's Theorem without the Axiom of Choice}
\author[T. Usuba]{Toshimichi Usuba}
\address[T. Usuba]
{Faculty of Science and Engineering,
Waseda University, 
Okubo 3-4-1, Shinjyuku, Tokyo, 169-8555 Japan}
\email{usuba@waseda.jp}
\keywords{Axiom of Choice, \L o\'s's theorem, Symmetric extension, Ultrapower, Ultraproduct}
\subjclass[2020]{Primary 03C20, 03E25, 03E35, 03E55}
\begin{document}

\begin{abstract}
We study some topics about \L o\'s's theorem without
assuming the Axiom of Choice. 
We prove that \L o\'s's fundamental theorem of ultraproducts is
equivalent to a weak form that 
every ultrapower is elementarily equivalent to
its source structure.
On the other hand, it is consistent that
there is a structure $M$ and an ultrafilter $U$
such that the ultrapower of $M$ by $U$ is elementarily equivalent to
$M$, but the fundamental theorem for the ultrapower of $M$ by $U$ fails.
We also show that weak fragments of the Axiom of Choice,
such as the Countable Choice,
do not follow 
from \L o\'s's theorem, even assuming 
the existence of non-principal ultrafilters.

\end{abstract}
\maketitle

\section{Introduction}
\emph{Ultraproduct} and \emph{ultrapower} are basic tools in model theory, set theory, and many other fields.
For an indexed family of first order structures $\{M_i \mid  i \in I\}$ with same language and an ultrafilter $U$ over $I$,
let $\prod_{i \in I} M_i/U$ denote the ultraproduct of the family $\{M_i \mid i \in I\}$
by $U$. 
For  ultraproducts, the following \emph{\L o\'s's theorem} is very important:
\begin{thm}[\L o\'s's  fundamental theorem of ultraproducts \cite{L}]
Let $U$ be an ultrafilter over a set $I$, and
$\{M_i \mid i\in I\}$ an indexed family of structures with language $\calL$.
For $f \in \prod_{i \in I} M_i$, let $[f]$ denote the 
equivalence class of $f$ modulo $U$.
Then for every $\calL$-formula $\varphi(v_0,\dotsc, v_n)$ and
$f_0,\dotsc, f_n \in \prod_{i \in I} M_i$,
\begin{align*}
\prod_{i \in I} M_i /U & \models \varphi([f_0],\dotsc, [f_n])\\
&\iff \{ i \in I \mid M_i \models \varphi(f_0(i),\dotsc, f_n(i))\} \in U.
\end{align*}
\end{thm}
This theorem provides considerable information about ultraproducts and ultrapowers;
For instance, ${}^IM/U$, the ultrapower of a first order structure $M$ by
an ultrafilter $U$ over a set $I$,  is elementarily equivalent to $M$.

\L o\'s's theorem is a theorem of $\ZFC$,
and it turned out that
\L o\'s's theorem is not provable from $\ZF$ alone:
Let $\mathsf{BPI}$ be the statement that every Boolean algebra has a
prime ideal.
Howard \cite{H} showed that \L o\'s's theorem$\,+\,\mathsf{BPI}$ is equivalent to $\mathsf{AC}$
in $\ZF$.
It is known that $\ZF+\mathsf{BPI}$ does not imply $\mathsf{AC}$,
so \L o\'s's theorem is not provable from $\ZF$.

Moreover, it was appeared that \L o\'s's theorem is strictly weaker than $\mathsf{AC}$.
Blass \cite{B} constructed a model of $\ZF$ in which
there is no non-principal ultrafilter.
In his model, $\mathsf{AC}$ fails but
\L o\'s's theorem holds (in the trivial sense).
Hence in $\ZF$, \L o\'s's theorem does not imply $\mathsf{AC}$.

Here we focus on the statement that
``every ultrapower is elementarily equivalent to its source structure''.
As stated before,
this is just a consequence of \L o\'s's theorem,
and it seems weaker than \L o\'s's theorem.
However, some papers and books (e.g., Howard \cite{H}, Howard-Rubin \cite{AC})
present this weak statement as \L o\'s's theorem.
Actually Howard's proof was carried out with this weak statement,
and many consequences of \L o\'s's theorem follow from
this weak statement (e.g., see Tachtsis \cite{T}).
This situation arises the following natural question, which was also asked in \cite{T}.
\begin{question}[In $\ZF$]
Is the weak statement that ``every ultrapower is elementarily equivalent to its source structure''
really weaker than \L o\'s's fundamental theorem of ultraproducts?
\end{question}

To distinguish various forms of \L o\'s's theorem,
let us  introduce  the following definitions:
Let $U$ be an ultrafilter over a set $I$.
Let $M$ be a first order structure,
and $\{M_i \mid i\in I\}$ an indexed family of first order structures with same language.
\begin{enumerate} \item $\LT 1(M, U)$: $M$ is elementarily equivalent to ${}^I M/U$.
\item $\LT 2(M, U)$: The fundamental theorem holds for the ultrapower 
${}^I M/U$.
\item $\LT 3(\{M_i \mid i \in I\}, U)$: If $\prod_{i \in I} M_i \neq \emptyset$, then
the fundamental theorem holds for the ultraproduct
$\prod_{i \in I} M_i/U$.
\end{enumerate}
Let $\LT  n$ for $n =1,2$ be the statement that
$\LT  n(M, U)$ holds for every structure $M$ and ultrafilter $U$ over $I$,
and $\LT 3$ be 
that $\LT 3(\{M_i \mid i \in I\}, U)$ holds for every family $\{M_i \mid i \in I\}$ of structures with
same language and every ultrafilter $U$ over $I$.
$\LT 3$ is \L o\'s's  fundamental theorem of ultraproducts,
and $\LT 1$ is the weak statement in question.
The following implications are clear from the definitions:
\[\LT 2(M, U) \Rightarrow \LT 1(M, U)\]
and
\[ \LT 3\Rightarrow \LT 2 \Rightarrow \LT 1.\]

In this paper, first we show  the following theorem in Section \ref{sec3}:
\begin{thm}[In $\ZF$]\label{1.3}
$\LT 1$, $\LT 2$, and $\LT 3$ are equivalent.
\end{thm}
Because of this theorem,
we can let $\LT$ denote $\LT 1$, $\LT 2$, and/or $\LT 3$.

We now know that $\LT 1$ implies $\LT 2$,
but the difference between $\LT 1(M, U)$ and $\LT 2(M, U)$ is worthy of attention,
because Spector \cite{Sp} proved that  if $M$ is a model of $\ZF$\footnote{Not necessary an $\in$-model.},
then $\LT 1(M, U)$ is equivalent to $\LT 2(M,U)$.
This result suggests that an implication
$\LT 1 \Rightarrow \LT 2$ is in fact \emph{uniform},
that is, the implication $\LT 1(M, U)\Rightarrow \LT 2(M,U)$ always holds for every $M$ and $U$.
In Section \ref{sec4}, however,  we prove that an implication $\LT 1 \Rightarrow \LT 2$ is not uniform,
in fact this uniform implication is equivalent to $\LT$.
\begin{thm}[In $\ZF$]\label{1.6}
\begin{enumerate}
\item If there is a non-principal ultrafilter $U$ over $\om$ but
the Countable Choice fails,
then there is a first order structure $M$
such that $\LT 1(M, U)$ holds but
$\LT 2(M, U)$ fails.
\item $\LT$ is equivalent to the following statement:
For every first order structure $M$ and every ultrafilter $U$,
if $\LT 1(M, U)$ holds then $\LT 2(M, U)$ holds as well.
\end{enumerate}
\end{thm}

In Sections \ref{sec5} and \ref{sec6}, we reexamine the fact that $\LT$ does not imply $\AC$.
Blass's model satisfies $\LT$ but the Countable Choice fails in his model. 
However there is no non-principal ultrafilter in the model, 
and $\LT$ only holds in the trivial sense.
This observation lead us to ask the following question.
\begin{question}[In $\ZF$]
Does $\LT+\mathsf{WUF}$ or $\LT+\mathsf{WUF(?)}$ deduce (some fragment of) $\AC$?
Where  $\mathsf{WUF}$ is the weak ultrafilter theorem that every infinite set carries a 
non-principal ultrafilter,
and $\mathsf{WUF(?)}$ is  the statement that there is a non-principal ultrafilter over some set\footnote{This notation is due to Herrlich \cite{He}.}.
\end{question}
Concerning this question, Tachtsis \cite{T} 
proved that $\LT+\mathsf{WUF}$
(in fact $\LT+$``there is a non-principal ultrafilter over $\om$'') implies the Countable Choice, and
he asked if
$\LT + \mathsf{WUF}$ implies $\mathsf{AC}_{\mathrm{fin}}$,
where $\mathsf{AC}_{\mathrm{fin}}$ is the statement that
every family of non-empty finite sets has a choice function.

For this question, in Sections \ref{sec5} and \ref{sec6} we provide two answers.
First answer is the following.
For a cardinal $\ka$ and a natural number $n$,
let $\mathsf{AC}_{\ka,n}$ be the statement that
every indexed family $\{A_\alpha \mid \alpha<\ka\}$
of non-empty sets with $\size{A_\alpha} \le n$ has a choice function.

\begin{thm}\label{1.8}
If $\ZFC$ is consistent, then so is
$\ZF+\LT+\mathsf{WUF}+\neg \AC_{\om_1,2}$.
\end{thm}
Hence $\LT+\mathsf{WUF}$
does not imply $\AC_{\om_1,2}$. This gives an answer to Tachtsis's question.

We give a second answer by using large cardinals:
\begin{thm}\label{1.9}
If $\ZFC+$``there is a measurable cardinal'' is consistent, then so is
$\ZF+\LT+\neg \AC_{\om,2}+$``there is a measurable cardinal''$+$
``every ultrafilter is $\sigma$-complete''.
\end{thm}
Thus $\LT+\mathsf{WUF(?)}$ does not imply $\AC_{\om,2}$, under the consistency of large cardinals.

In Section \ref{sec7}, we prove the following result, which contrasts with Theorem \ref{1.8} in the view point of the consistency strength.
It also shows that the large cardinal assumption in Theorem \ref{1.9} cannot be eliminated.
Let $\mathsf{AC}_\om$ denote the Countable Choice.

\begin{thm}\label{1.8+}
The following theories are equiconsistent:
\begin{enumerate}
\item $\ZF+\LT+\mathsf{WUF(?)}+\neg \AC_{\om,2}$.
\item $\ZF+\LT+\mathsf{WUF(?)}+\neg \AC_\om$.
\item $\ZFC+$``there is a measurable cardinal''.
\end{enumerate}
\end{thm}

We also consider the statement that $\mathsf{WUF(?)}+$``every ultrafilter is $\sigma$-complete''.
We prove that this statement is equiconsistent with $\ZFC$,
but the statement with $\LT$ has a large cardinal strength.

\section{Preliminaries}
{\bf From now on, our base theory is $\ZF$ unless otherwise specified.}
Throughout this paper, a \emph{filter over a set $I$} (ultrafilter, respectively)
will be a proper filter (proper ultrafilter, respectively)
on the Boolean algebra $\p(I)$. An element of an ultrafilter is called \emph{a measure one set}.

A \emph{structure} is a first order structure with some language $\calL$,
and we sometimes identify each symbol in $\calL$ with its interpretation in $M$.
In Sections \ref{sec6} and \ref{sec7}, a set (or class) $X$ will be identified 
with $\in$-structure $\seq{X; \in}$.

Here we recall ultraproducts and ultrapowers.
Let $I$ be a set and $U$ an ultrafilter over $I$.
Let $\{M_i \mid i \in I\}$ be an indexed  family of structures with same language $\calL$.
For $f,g \in \prod_{i \in I} M_i$,
define $f \simeq g$ if $\{i \in I \mid f(i) = g(i) \} \in U$,
this is an equivalence relation on $\prod_{ i \in I} M_i$.
For $f \in \prod_{i \in I} M_i$,
let $[f]$ be the equivalence class of $f$.
The \emph{ultraproduct of $\{M_i \mid i \in I\}$ by $U$},
denoted by $\prod_{ i \in I} M_i/U$, is the structure with language $\calL$ defined as in the following manner.
The base set of the ultraproduct is the set $\{[f] \mid f \in \prod_{i \in I} M_i\}$.
For an $n$-array relation symbol $R \in \calL$,
its interpretation $R^*$ in $\prod_{i \in I} M_i/U$ is defined as:
\[
R^*([f_0],\dotsc, [f_{n}]) \iff \{i \in I \mid M_i \models R(f_0(i),\dotsc, f_{n}(i))\} \in U.
\] 
We define the interpretations of function symbols and constant symbols by a similar manner.

If every $M_i$ is the same structure $M$,
let us denote $\prod_{i \in I} M_i/U$ by ${}^I M/U$, it is called the \emph{ultrapower of $M$ by $U$}.

Here we state various forms of \L o\'s's theorem explicitly.
\begin{define}
Let $U$ be an ultrafilter over a set $I$.
Let $M$ be a first order structure with language $\calL$,
and $\{M_i \mid i\in I\}$ an indexed family of first order structures with language $\calL$.
Let $\LT 1(M, U)$, $\LT 2(M, U)$, 
and $\LT 3(\{M_i \mid i \in I\}, U)$ be the following statements:
\begin{enumerate}
\item $\LT 1(M, U)$: $M$ is elementarily equivalent to ${}^I M/U$.
\item $\LT 2(M, U)$: For every $\calL$-formula $\varphi(v_0,\dotsc, v_n)$ and 
$f_0,\dotsc, f_n \in {}^I M$,
\begin{align*}
{}^I M /U & \models \varphi([f_0],\dotsc, [f_n])\\
& \iff \{i \in I \mid M \models \varphi(f_0(i),\dotsc, 
f_n(i) )\} \in U.
\end{align*}
\item $\LT 3(\{M_i \mid i \in I\}, U)$: If $\prod_{i \in I} M_i \neq \emptyset$, then
for every  $\calL$-formula $\varphi(v_0,\dotsc, v_n)$ and $f_0,\dotsc, f_n  \in \prod_{i \in I} M_i$,
\begin{align*}
\prod_{i \in I} M_i /U & \models \varphi([f_0],\dotsc, [f_n])\\
& \iff \{i \in I \mid M_i \models \varphi(f_0(i),\dotsc, 
f_n(i) )\} \in U.
\end{align*}
\end{enumerate}
Let $\LT  n$ for $n =1,2$ be the statement that
$\LT  n(M, U)$ holds for every structure $M$ and ultrafilter $U$,
and $\LT 3$ be 
that $\LT 3(\{M_i \mid i \in I\}, U)$ holds for every family $\{M_i \mid i \in I\}$ of structures with
same language and every ultrafilter $U$.
\end{define}
We will show that $\LT 1$, $\LT 2$, and $\LT 3$ are equivalent.

A \emph{cardinal} is a well-orderable cardinal.
An ultrafilter $U$ over an infinite cardinal $\ka$ is \emph{uniform} if every
measure one set of $U$ has cardinality $\ka$.

For an infinite cardinal $\ka$,
an ultrafilter $U$ is \emph{$\ka$-complete}
if for every $\alpha<\ka$ and $f:\alpha \to U$,
we have $\bigcap_{\beta<\alpha} f(\beta) \in U$.
$U$ is \emph{$\sigma$-complete} if it is $\om_1$-complete.
Every ultrafilter is $\om$-complete,
and if $U$ is a principal ultrafilter, then $U$ is $\ka$-complete for every cardinal $\ka$.
For an ultrafilter $U$, if $\ka$ is the least cardinal such that $U$ is not $\ka$-complete,
then $\ka$ must be a successor cardinal.

Let $S$, $T$ be sets, and $U$ an ultrafilter over a set $S$.
For a map $f:S \to T$,
put $f_*(U)=\{X \subseteq T \mid f^{-1}(X) \in U\}$.
$f_*(U)$ is an ultrafilter over $T$,
and if $U$ is $\ka$-complete for some cardinal $\ka$,
then $f_*(U)$ is also $\ka$-complete.

\begin{lemma}\label{6.22}
Let $U$ be an ultrafilter over a set $S$ and $\ka$ a cardinal.
Suppose $\ka$ is the least cardinal such that 
there is $f:\ka \to U$ with $\bigcap_{\alpha<\ka} f(\alpha) \notin U$,
that is, $U$ is $\ka$-complete but not $\ka^+$-complete.
Then $\ka$ is regular and there is a function $g:S \to \ka$ 
such that $g^{-1}(\alpha) \notin U$ for every $\alpha<\ka$.
\end{lemma}
\begin{proof}
It is easy to show that $\ka$ is regular.
To find a required function $g$,
we may assume $\bigcap_{\alpha<\ka} f(\alpha) =\emptyset$.
For $\gamma<\ka$, let $X_\gamma=\bigcap_{\alpha<\gamma} f(\alpha)$.
By the minimality of $\ka$, we have that $X_\gamma \in U$,
and $\bigcap_{\gamma<\ka} X_\gamma=\emptyset$. 
Now define $g:S \to \ka$ as in the following manner.
For $s \in S$, let $\gamma_s$ be the least ordinal $\gamma$ with $s \notin X_\gamma$.
Let $g(s)=\gamma_s$.
It is easy to check that this $g$ works.
\end{proof}

An uncountable cardinal $\ka$ is a \emph{measurable cardinal}
if $\ka$ carries a $\ka$-complete non-principal ultrafilter.
Every measurable cardinal is regular.
In $\ZFC$,
there is a measurable cardinal if and only if
there is a non-principal $\sigma$-complete ultrafilter,
but this equivalence can fail in $\ZF$, see Section \ref{sec7}.

\begin{fact}
If there is a measurable cardinal,
then there is an inner model of $\ZFC+$``there is a measurable cardinal''.
\end{fact}


We recall weak forms of $\AC$.
For a cardinal $\ka$ and a natural number $n$,
let $\mathsf{AC}_{\ka,n}$ be the statement that
every indexed family $\{A_\alpha \mid \alpha<\ka\}$
of non-empty sets with $\size{A_\alpha} \le n$ has a choice function.
$\mathsf{AC}_\ka$ be the statement that
every indexed family $\{A_\alpha \mid \alpha<\ka\}$
of non-empty sets has a choice function.
$\mathsf{AC}_\om$ is called \emph{the Countable Choice}.
\emph{The weak ultrafilter theorem}, denoted by
$\mathsf{WUF}$, is the statement that
every infinite set carries a non-principal ultrafilter.
$\mathsf{WUF(?)}$ is that there is a non-principal ultrafilter over some set.
Clearly $\mathsf{WUF}$ implies $\mathsf{WUF(?)}$.

For a set $X$ and an inner model $M$ of $\ZF$,
let $M(X)$ be the inner model of $\ZF$ containing $X$ and 
all elements of $M$,
namely, $M(X)=\bigcup_{\alpha}L(M_\alpha \cup \mathrm{trcl}(\{X\}))$
where $M_\alpha$ is the set of all sets in $M$ with rank $<\alpha$.
The following follows from the construction of $M(X)$.
\begin{fact}\label{model}
Let $M$ be an inner model of $\ZF$ and $X$ a set.
If $M$ satisfies $\AC$, then for every set $S \in M(X)$,
there is an ordinal $\gamma$ and a surjection $f \in M$
from ${}^{<\om}\gamma  \times {}^{<\om} \mathrm{trcl}(\{X\})$ onto $S$.
\end{fact}

\section{\L o\'s's Theorem}\label{sec3}
In this section we prove that $\LT 1$, $\LT 2$, and $\LT 3$ are equivalent.
First we prove the following  useful lemma, a similar result was proved by Spector \cite{Sp}.

\begin{lemma}\label{2.1+}
Let $U$ be an ultrafilter over a set $I$.
Then the following are equivalent:
\begin{enumerate}
\item $\LT 3(\{M_i \mid i \in I\}, U)$ holds for every
indexed family $\{M_i \mid i \in I\}$ of structures with same language.
\item For every indexed family $\{A_i \mid i \in I\}$ of non-empty sets,
there is a function $g$ on $I$ such that
$\{i \in I \mid g(i) \in A_i\} \in U$.
\end{enumerate}
\end{lemma}
\begin{proof}
(2) $\Rightarrow$ (1) follows from a standard proof of \L o\'s's fundamental theorem of ultraproducts;
To carry out a proof of \L o\'s's theorem without $\mathsf{AC}$,
a problematic part is the direction that
\begin{align*}
\{ i \in I \mid M_i  \models \exists v \varphi& (f_0(i),\dotsc, f_n(i), v)\} \in U\\
&  \Rightarrow \prod_{i \in I} M_i /U \models \exists v \varphi([f_0],\dotsc, [f_n],v)
\end{align*}
for $f_0,\dotsc, f_n \in \prod_{i \in I} M_i$ and a formula $\varphi(v_0,\dotsc, v_n, v)$.
Now suppose 
$\{ i \in I \mid M_i  \models \exists v \varphi (f_0(i),\dotsc, f_n(i), v)\} \in U$.
Define $A_i$ as: If
$M_i \models \exists v \varphi (f_0(i),\dotsc, f_n(i), v)$,
then let $A_i=\{ x \in M_i \mid 
M_i \models  \varphi (f_0(i),\dotsc, f_n(i), x)\}$.
Otherwise, let $A_i$ be $M_i$.
By (2), we can find a function $g$ on $I$ such that
$\{i \in I \mid g(i) \in A_i\}\in U$.
Clearly $[g]$ witnesses that
$\prod_{i \in I} M_i /U \models \exists v \varphi([f_0],\dotsc, [f_n],v)$.

(1) $\Rightarrow$ (2).
For a given family $\{A_i \mid i \in I\}$,
fix a point $p \notin \bigcup_{i \in I} A_i$,
and let $B_i=A_i \cup\{ p\}$.
We identify each $B_i$ with some structure.
Since $\prod_{i \in I} B_i \neq \emptyset$,
we can apply (1) to the ultraproduct
$\prod_{i \in I} B_i /U$.
For each $B_i$, obviously $\exists v (v \neq p)$ holds.
Let $c_p:I \to \{p\}$ be the constant function.
By (1),
we know $\prod_{i \in I} B_i /U \models \exists v (v \neq [c_p])$,
hence there is $[g] \in \prod_{i \in } B_i/U$
such that $\prod_{i \in I} B_i /U \models [g] \neq [c_p]$.
Then $\{i \in I \mid g(i) \neq p \} \in U$,
so $\{ i \in I \mid g(i) \in A_i\} \in U$.
\end{proof}

Now we prove Theorem \ref{1.3}.
\begin{thm}\label{2.2+}
$\LT 1$ implies $\LT 3$.
\end{thm}
\begin{proof}
Suppose $\LT 1$.
Fix an ultrafilter $U$ over $I$.
By Lemma \ref{2.1+}, it is sufficient to show that
for every indexed family $\{A_i \mid i \in I\}$ of non-empty sets,
there is a function $g$ on $I$ with
$\{i \in I \mid g(i) \in A_i\} \in U$.

Let $S=I \cup \bigcup_{i \in I} A_i$.
Define the unary relation $R_1 \subseteq S$ and
the binary relation $R_2 \subseteq S^2$ as follows:
$R_1(x) \iff x \in I$, and
$R_2(x,y) \iff x \in I$ and $y \in A_x$.
Let $M$ be the structure $\seq{S;R_1, R_2}$.
Let $M^*=\seq{S^*;R^*_1,R_2^*}$ be the ultrapower of $M$ by $U$,
where $R^*_1 \subseteq S^*$ and $R^*_2\subseteq (S^*)^2$ are relations corresponding to $R_1$ and $R_2$ respectively.

By the definition of $R_1$ and $R_2$,
we have $M \models \forall x(R_1(x) \to \exists y R_2(x,y))$.
By $\LT 1$,
we also have 
$M^* \models \forall x (R_1^*(x) \to \exists y R_2^*(x,y))$.
Let $\mathrm{id}:I \to I$ be the identity function on $I$.
Since $M \models R_1(\mathrm{id}(i))$ for every $i \in I$,
we have $M^* \models R^*_1 ([\mathrm{id}])$.
Hence there is $[g] \in M^*$ such that
$M^* \models R^*_2([\mathrm{id}], [g])$.
By the definition of the ultrapower,
we have
$\{i \in I \mid M \models R_2(i, g(i))\} \in U$.
Then 
$\{ i \in I \mid g(i) \in A_i\} \in U$
 by the definition of $R_2$.
\end{proof}

From now on, we let $\LT$ denote $\LT 1$, $\LT 2$, and/or $\LT 3$.
This does not cause any problem by the theorem.

The following are immediate from the proof of Theorem \ref{2.2+}.
\begin{cor}\label{2.33}
$\LT$ is equivalent to the following weak form of $\LT 1$:
For every ultrafilter $U$ over $I$ and 
structure $M$ with relational language,
$\LT 1(M, U)$ holds.
\end{cor}

\section{$\LT 1(M, U)$ and $\LT 2(M,U)$}\label{sec4}
In this section, we show that, under some assumption,
there are $M$ and $U$ such that $\LT 1(M, U)$ holds but $\LT 2 (M, U)$ fails.
First we prove it under the assumption that
$\om$ carries a non-principal ultrafilter  but 
the Countable Choice fails. This is (1) of Theorem \ref{1.6}.

\begin{thm}\label{2.1}
Suppose there is a non-principal ultrafilter $U$ over $\om$
but the Countable Choice fails.
Then there is a structure $M$
such that
$\LT 1(M, U)$ holds but
$\LT 2(M, U)$ fails.
\end{thm}
Note that the assumption of this theorem is consistent;
In the first Cohen model, the Countable Choice fails but $\mathsf{BPI}$ holds
(see \cite{AC}).
We also note that 
by this theorem,
if $\LT$ holds and $\om$ carries a non-principal ultrafilter then
the Countable Choice holds,
and this result already proved by Tachtsis (\cite{T}).

\begin{proof}
Take an indexed family $\{E_n \mid n<\om\}$ of non-empty sets 
which has no choice function.
We may assume that $E_0$ is infinite and $\om \cap \bigcup_{n<\om} E_n =\emptyset$.

\begin{claim}\label{4.22}
There is an infinite set $S$, $A \subseteq S$, and $B_x \subseteq S$ for $x \in A$
satisfying the following:
\begin{enumerate}
\item $A$ is countably infinite.
\item For each $x \in A$, the set $B_x$ is infinite.
\item For distinct $x, x' \in A$, $B_x \cap B_{x'}=\emptyset$.
\item $A \cap \bigcup_{x \in A} B_x=\emptyset$.
\item $S \setminus (A \cup \bigcup_{x \in A} B_x)$ 
is infinite.
\item For every infinite $A' \subseteq A$,
the subfamily $\{B_x \mid x \in A'\}$ has no choice function.
\end{enumerate}
\end{claim}
\begin{proof}
Let $A=\om$, $B_x=\{x\} \times \prod_{n \le x} E_n$ for $x \in A$,
and $S=(\om+\om) \cup \bigcup_{x \in A} B_x$.
We check that (6) holds.
Let $A'$ be an infinite subset of $A$ ($=\om)$ and
suppose to the contrary that $\{B_x \mid x \in A'\}$ has a choice function $f$.
Define $g(n)$ as follows:
Pick the minimal $x \in A'$ with $x\ge n$,
and let $f(x)=\seq{x, \seq{e_0,\dotsc, e_x}}$.
Then let $g(n)=e_n$.
It is clear that $g$ is a choice function of $\{E_n \mid n<\om\}$,
this is a contradiction.
\end{proof}

Let $S, A, B_x$ be from the claim.
Define the binary relation $R \subseteq S^2$ by $R(x,y) \iff x \in A$ and $y \in B_x$.
Let $M$ be the structure $M=\seq{S; R}$, and
$M^*=\seq{S^*; R^*}$  the ultrapower of $M$ by $U$.

\begin{claim}\label{2.3}
There is a function $f:\om \to S$ and a formula $\varphi(x)$
such that $M \models \varphi(f(n))$ for every $n<\om$,
but  $M^* \models \neg \varphi([f])$. Hence $\mathsf{LT2}(M, U)$ fails.
\end{claim}
\begin{proof}
Since $A$ is countably infinite, we can take a bijection $f:\om \to A$.
Consider the formula $\varphi(v)$ that $\exists w R(v,w)$.
For each $n<\om$, 
we have that $\exists w R(f(n), w)$ holds in $M$.
If $M^* \models \exists w \varphi([f],w)$, then there is $[g] \in S^*$
with $M^* \models R^*([f], [g])$.
So $X=\{n \mid R(f(n), g(n))\}=\{n \mid g(n) \in B_{f(n)}\}  \in U$,
and $X$ is infinite.
Then the set $A'=\{f(n) \mid n \in X\}$ is infinite and 
$g$ induces  a choice function of $\{B_{x} \mid x \in A'\}$,
this contradicts (6) of Claim \ref{4.22}.
\end{proof}

Let $C^*=\{x \in S^* \mid \exists w R^*(x,w)\}$,
and for $x \in C^*$, let $D_x^*=\{y \in S^* \mid R^*(x,y)\}$.
\begin{claim}
\begin{enumerate}
\item For $[f] \in S^*$, if $[f] \in C^*$ then
$\{n \mid f(n) \in A \} \in U$.
\item For distinct $x, x' \in C^*$, 
we have $D_x^* \cap D_{x'}^*=\emptyset$.
\item $C^* \cap \bigcup_{x \in C^*} D_x^*=\emptyset$.
\end{enumerate}
\end{claim}
\begin{proof}
(1) Since $[f] \in C^*$,
there is $[g] \in S^*$ with $R^*([f], [g])$.
Then $X=\{n \mid R(f(n), g(n)) \} \in U$.
For $n \in X$, we know $R(f(n), g(n))$, thus $f(n) \in A$.
Then $X \subseteq \{n \mid f(n) \in A\}$, and $\{n \mid f(n) \in A\} \in U$.

(2)  Let $[f]=x$ and $[f']=x'$. 
If there is $[g] \in D_x^* \cap D_{x'}^*$,
then $Y=\{n \mid R(f(n), g(n)), R(f'(n), g(n))\} \in U$.
For $n \in Y$, we know $g(n) \in B_{f(n)} \cap B_{f'(n)}$, so
$f(n)=f'(n)$ must hold. Hence $Y \subseteq \{n \mid f(n)=f'(n) \} \in U$,
and $[f]=[f']$.
This is a contradiction.

(3) is similar.
\end{proof}
Note that the converse of (1) does not hold;
If $f:\om \to A$ is a bijection,
then $\{n \mid f(n) \in A\} \in U$ but $[f] \notin C^*$ by the proof of Claim \ref{2.3}.

For $z \in S$, let $c_z:\om \to \{z\}$ be the constant function.
The following is easy to check.
\begin{claim}\label{2.5}
If $z \in A$ then $[c_z] \in C^*$. In particular the assignment $z \mapsto [c_z]$ is an 
injection from $A$ into $C^*$.
\end{claim}
\begin{claim}\label{2.6}
$S^* \setminus (C^* \cup \bigcup_{x \in C^*} D_x^*)$ is infinite.
\end{claim}
\begin{proof}
If $z \in S \setminus(A \cup \bigcup_{x \in A} B_x)$
then it is easy to show that $[c_z] \notin (C^* \cup \bigcup_{x \in C^*} D_x^*)$.
Since $S \setminus(A \cup \bigcup_{x \in A} B_x)$ is infinite,
we have that $S^* \setminus (C^* \cup \bigcup_{x \in C^*} D_x^*)$ is infinite as well.
\end{proof}

\begin{claim}\label{2.7}
For every $x \in C^*$,
there is $a \in A$ with
$[c_a]=x$.
\end{claim}
\begin{proof}
Let $x=[f]$ and suppose to the  contrary that there is no $a \in A$ with
$\{n \mid f(n)=c_a(n)=a\} \in U$.
Pick $[g] \in S^*$ with $R^*([f], [g])$.
We know $Y=\{n \mid f(n) \in A, R(f(n), g(n))\} \in U$. 
Then we can take a sequence $n_0<n_1<\cdots$ from $Y$
such that $f(n_k) \neq  f(n_l)$ for every $k<l<\om$.
The map $g$ induces a choice function of the family $\{B_x \mid x \in \{f(n_k) \mid k<\om\} \}$,
this is a contradiction.
\end{proof}

\begin{claim}\label{2.8}
Let $a \in A$ and $x=[c_a] \in C^*$.
Then $D_x^*$ is infinite.
\end{claim}
\begin{proof}
Take an arbitrary $m<\om$.
Pick $m$ many elements $b_0,\dotsc, b_{m-1} \in B_a$.
Then $R(a, b_k)$ holds for every $k < m$,
hence $R^*([c_a], [c_{b_k}])$ holds,
and $[c_{b_0}],\dotsc, [c_{b_{m-1}}] \in D_x^*$.
\end{proof}

Finally we prove $\mathsf{LT1}(M, U)$ holds.
\begin{claim}
$M$ is elementarily equivalent to $M^*$.
\end{claim}
\begin{proof}
Take a generic extension $V[G]$ in which both $S$ and $S^*$ are countable.
In $V[G]$, we shall construct an isomorphism from $M$ onto $M^*$.
Then $M$ is elementarily equivalent to $M^*$ in $V[G]$.
On the other hand, the satisfaction relations on $M$ and $M^*$ are absolute between $V$ and $V[G]$,
hence we conclude that $M$ is elementarily equivalent to $M^*$ in $V$.

We work in $V[G]$. Note  the following observations:
\begin{enumerate}
\item The assignment $a \mapsto [c_a]$ gives a bijection from $A$ onto $C^*$ by Claims \ref{2.5} and \ref{2.7}.
\item Since $S \setminus (A \cup \bigcup_{x \in A} B_x)$ and 
$S^* \setminus (C^* \cup \bigcup_{x \in C^*} D_x^*)$ are countably infinite in $V[G]$, 
there is a bijection between them.
\item By Claim \ref{2.8}, 
we have both $B_a$ and $D^*_{[c_a]}$ are countably infinite for every $a \in A$.
Because $S$ and $S^*$ are countable in $V[G]$,
for each $a \in A$, we can take a bijection between $B_a$ and $D^*_{[c_a]}$
in some uniform way.
\end{enumerate}

By these observations,
we can take a bijection $j:S \to S^*$
such that:
\begin{enumerate}
\item $j``A=C^*$.
\item $j``B_x=D^*_{j(x)}$ for every $x \in A$.
\item $j``(S \setminus (A \cup \bigcup_{x \in A} B_x))=
S^* \setminus (C^* \cup \bigcup_{x \in C^*} D_x^*)$.
\end{enumerate}
Then for every $x, y \in S$,
\[
R(x,y) \iff y \in B_x \iff j(y) \in D_{j(x)}^* \iff R^*(j(x), j(y)).
\]
Thus  $j$ is an isomorphism from $M$ onto $M^*$.
\end{proof}
\end{proof}

\begin{note}
Under $\LT$, the Countable Choice follows from the existence of a non-principal ultrafilter over $\om$,
but the converse does not hold:
Pincus-Solovay \cite{PS} constructed a model of $\ZF$ in which the Dependent Choice holds but
there is no non-principal ultrafilter.
In their model, $\LT$ holds trivially.
\end{note}

Next, we prove that if $\LT$ fails 
then $\LT 1(M, U)$ holds but $\LT 2(M,U)$ fails for some $M$ and $U$.
This immediately shows (2) of Theorem \ref{1.6}.

\begin{thm}\label{4.111}
Suppose $\LT$ fails.
Then there is a structure $M$ and 
an ultrafilter $U$ over a set $I$ such that
$\LT 1(M, U)$ holds but $\LT 2(M, U)$ fails.
\end{thm}
\begin{proof}
By Lemma \ref{2.1+}, there
is an indexed family 
$\{E_i \mid i \in I\}$ 
of non-empty sets
such that
there is no function $g$ on $I$ with
$\{i \in I \mid g(i) \in E_i\} \in U$.
By using this family, we shall construct a structure $M$ and an ultrafilter $U$ 
such that $\LT 1(M, U)$ holds but
$\LT 2(M, U)$ fails as in the proof of Theorem \ref{2.1}.

First, we may assume that every $i \in I$ is neither an ordinal nor an ordered pair.
Let $A=I \cup \om$.
%
\begin{claim}\label{4.1111}
There is an infinite set $S$ and  non-empty $B_x \subseteq S$ for $x \in A$
satisfying the following:
\begin{enumerate}
\item $A \subseteq S$.
\item The set $\{x  \in A \mid B_x$ is infinite$\}$ is infinite.
\item For every positive $n \in \om$, the set $\{x  \in A \mid \size{B_x}=n\}$ is infinite. 
\item For distinct $x, x' \in A$, $B_x \cap B_{x'}=\emptyset$.
\item $A \cap \bigcup_{x \in A} B_x=\emptyset$.
\item $S \setminus (A \cup \bigcup_{x \in A} B_i)$ 
is infinite.
\item There is no function $g$ with $\dom(g)=I$ and
$\{ i \in I \mid g(i) \in B_i\} \in U$.
\end{enumerate}
\end{claim}
\begin{proof}
Fix an indexed family $\{E_k \mid k<\om\}$ of non-empty sets
such that:
\begin{itemize}
\item The set $\{k<\om \mid E_k$ is infinite$\}$ is infinite.
\item For every positive $n \in \om$,
the set $\{k<\om \mid \size{E_k}=n\}$ is infinite.
\end{itemize}
For $x \in A=I \cup \om$, let $B_x=\{x\} \times E_x$.
Then take a large set $S \supseteq A \cup \bigcup_{x \in A} B_x$
such that
$S \setminus (A \cup \bigcup_{x \in A} B_x)$ 
is infinite. It is easy to check that
$S$ and the $B_x$'s are as required.
\end{proof}
Define $R \subseteq S^2$ by $R(x, y) \iff x \in A$ and $y \in B_x$.
Let $M$ be the structure $\seq{S;R}$, and
$M^*=\seq{S^*; R^*}$ the ultrapower of $M$ by $U$.

\begin{claim}
There is a function $f:I \to S$ and a formula $\varphi(v)$
such that $M \models \varphi(f(i))$ for every $i \in I$,
but  $M^* \models \neg \varphi([f])$. Hence $\mathsf{LT2}(M, U)$ fails.
\end{claim}
\begin{proof}
As in the proof of Claim \ref{2.3},
consider the formula $\varphi(v)$ that $\exists w R(v,w)$.
Let $\mathrm{id}$ be the identity function on $I$.
For each $i \in I$,
we have that $\exists w R(\mathrm{id}(i), w)$ holds in $M$.
If $M^* \models \exists w \varphi([\mathrm{id}],w)$, then there is $[g] \in S^*$
with $M^* \models R^*([\mathrm{id}], [g])$.
Then $\{i \in I \mid R(\mathrm{id}(i), g(i))\}=\{i \in I  \mid g(i) \in B_{i}\}  \in U$,
this contradicts to (7) of Claim \ref{4.1111}.
\end{proof}

Let $C^*=\{x \in S^* \mid \exists w R^*(x,w)\}$,
and $D_x^*=\{y \in S^* \mid R^*(x,y)\}$ for $x \in C^*$.
The following claims can be verified easily.
\begin{claim}
\begin{enumerate}
\item For $[f] \in S^*$, if $[f] \in C^*$ then
$\{i \in I  \mid f(i) \in A \} \in U$.
\item For distinct $x, x' \in C^*$, 
we have $D_x^* \cap D_{x'}^*=\emptyset$.
\item $C^* \cap \bigcup_{x \in C^*} D_x^*=\emptyset$.
\end{enumerate}
\end{claim}

\begin{claim}
$S^* \setminus (C^* \cup \bigcup_{x \in C^*} D_x^*)$ is infinite.
\end{claim}

Recall that $c_z:I \to \{z\}$ is the constant function.
\begin{claim}
If $z \in A$ then $[c_z] \in C^*$. In particular the assignment $z \mapsto [c_z]$ is an 
injection from $A$ into $C^*$, and $C^*$ is infinite.
\end{claim}

\begin{claim}
Let $a \in A$ and $x=[c_a] \in C^*$.
\begin{enumerate}
\item If $B_a$ is infinite, then $D^*_x$ is infinite as well.
\item If $B_a$ is finite, then $D_x^*$ is finite and $\size{B_x}=\size{D_x^*}$.
\end{enumerate}
\end{claim}

By the claims and the choice of the $B_x$'s, we have:
\begin{claim}
\begin{enumerate}
\item The set $\{x \in C^* \mid D^*_x$ is infinite$\}$ is infinite.
\item For every positive $n \in \om$,
the set $\{x \in C^* \mid \size{D^*_x}=n\}$ is infinite.
\end{enumerate}
\end{claim}


\begin{claim}
$M$ is elementarily equivalent to $M^*$.
Hence $\LT 1(M, U)$ holds.
\end{claim}
\begin{proof}
We work in a generic extension $V[G]$ of $V$ in which
$S$ and $S^*$ are countably infinite.
By the claims above, we have:
\begin{enumerate}
\item $A$ and $C^*$ are countably infinite.
\item The sets $\{x \in A \mid B_x$ is countably infinite$\}$ and
$\{x \in C^* \mid D_x^*$ is countably infinite$\}$ are countably infinite.
\item For every positive $n<\om$,
the sets $\{x \in A \mid \size{B_x}=n\}$ and
$\{x \in C^* \mid \size{D_x^*}=n\}$ are countably infinite.
\item $S \setminus (A \cup \bigcup_{x \in A} B_x)$ and
$S^* \setminus (C^* \cup \bigcup_{x \in C^*} D_x^*)$ are countably infinite.
\end{enumerate}
Hence we can construct a bijection $j:S \to S^*$
such that 
\begin{enumerate}
\item $j``A=C^*$.
\item $j``B_x=D_{j(x)}^*$ for $x \in A$.
\item $j``(S \setminus (A \cup \bigcup_{x \in A} B_x))
=S^* \setminus (C^* \cup \bigcup_{x \in C^*} D_x^*)$.
\end{enumerate}
Then $j$ is an isomorphism between $M$ and $M^*$, hence $M$ is elementarily equivalent to $M^*$.
\end{proof}
\end{proof}

To conclude this section,
we see another weak form of \L o\'s's theorem.
Let $M$ be a structure and $U$ an ultrafilter over a set $I$.
The map $i:M \to {}^IM/U$ defined by
$i(x)=[c_x]$ is called the \emph{ultrapower embedding} where
$c_x$ is the constant function with value $x$.
Under $\LT$, every ultrapower embedding is an elementary embedding.
Now we let $\LT 1.5(M, U)$ be the statement that
the ultrapower embedding $i:M \to {}^IM/U$ is an elementary embedding.

It is clear that $\LT 2(M, U) \Rightarrow \LT 1.5(M, U) \Rightarrow \LT 1(M,U)$,
and  Spector \cite{Sp} showed that
$\LT 2(M, U)$, $\LT 1.5(M, U)$, and $\LT 1(M,U)$ are equivalent if $M$ is a model of $\ZF$.
Beside this result, we point out that,
if $M$ and $U$ are from the proof of Theorem \ref{4.111},
then $\LT 1.5(M,U)$ holds but $\LT 2(M,U)$ fails.
Here we sketch the proof.
Let $i:M \to {}^IM/U\,(=M^*)$ be the ultrapower embedding.
Take a formula $\varphi(v_0,\dotsc, v_n)$ and
$x_0,\dotsc, x_n \in M$.
In a generic extension $V[G]$ in which $S$ and $S^*$ are countable,
we can construct an isomorphism $j:M \to M^*$
such that $i(x_k)=j(x_k)$ for every $k\le n$.
Then 
\begin{align*}
M \models \varphi(x_0,\dotsc, x_n) &\iff 
M^* \models \varphi(j(x_0),\dotsc, j(x_n))\\
 &\iff
M^* \models \varphi(i(x_0),\dotsc, i(x_n)).
\end{align*}
So the ultrapower embedding $i$ is an elementary embedding.
This result also shows that
$\LT$ is equivalent to:
For every $M$ and $U$, if $\LT 1.5(M,U)$ holds then
$\LT 2(M,U)$ holds as well.

However we do not know the answer to the following question:
\begin{question}
\begin{enumerate}
\item Is it consistent that
there are $M$ and $U$ such
that $\LT 1(M,U)$ holds but $\LT 1.5(M,U)$ fails?
\item Is $\LT$ equivalent to the following statement?:
For every $M$ and $U$, if $\LT 1(M,U)$ holds then
$\LT 1.5(M,U)$ holds as well.
\end{enumerate}
\end{question}

\section{Choiceless models of $\LT$ (I)}\label{sec5}
To prove Theorems \ref{1.8} and \ref{1.9},
we will use \emph{symmetric extensions} (or \emph{symmetric submodels}).
Here we review it. See Jech \cite{J} for details,
and we follow Karagila's notations in \cite{Kar}.

Let $\bbP$ be a poset with maximum element $\mathbbm 1$.
For a set $x$, let $\check x$ be a canonical name for $x$,
namely, $\check x=\{\seq{\check y, \mathbbm 1} \mid y \in x\}$. 

Every automorphism $\pi$ on $\bbP$
induces the isomorphism $\pi$ on the $\bbP$-names,
namely, $\pi(\dot x)=\{\seq{\pi(\dot y), \pi(p)} \mid \seq{\dot y, p } \in \dot x\}$
for a $\bbP$-name $\dot x$.
\begin{fact}
Let $p \in \bbP$, 
$\varphi$ be a formula of set theory,
and $\dot x_0,\dotsc, \dot x_n$ $\bbP$-names.
Let $\pi$ be an automorphism on $\bbP$.
Then $p \Vdash \varphi(\dot x_0,\dotsc, \dot x_n)$ if and only if
$\pi(p) \Vdash \varphi(\pi(\dot x_0),\dotsc, \pi(\dot x_n))$.
\end{fact}

Let $\mathcal{G}$ be a subgroup of the automorphism group on $\bbP$.
A non-empty family $\calF$ of subgroups of $\calG$ is 
a \emph{normal filter on $\calG$} if the following hold:
\begin{enumerate}
\item If $H \in \calF$ and $H'$ is a subgroup of $\calG$ with
$H \subseteq H'$, then $H' \in \calF$.
\item For $H, H' \in \calF$, we have $H \cap H' \in \calF$.
\item For every $H \in \calG$ and $\pi \in \calG$,
the set $\pi^{-1} H\pi=\{\pi^{-1} \circ \sigma \circ \pi \mid \sigma \in H\}$ is in $\calF$.
\end{enumerate}
A triple $\seq{\bbP, \calG, \calF}$ is called a \emph{symmetric system}.

For a $\bbP$-name $\dot x$,
let $\mathrm{sym}(\dot x)=\{\pi \in \calG \mid \pi(\dot x)=\dot x\}$,
which is a subgroup of $\calG$.
A name $\dot x$ is \emph{symmetric}
if $\mathrm{sym}(\dot x) \in \calF$,
and \emph{hereditarily symmetric}
if $\dot x$ is symmetric and
for every $\seq{\dot y, p} \in \dot x$,
$\dot y$ is hereditarily symmetric.
\begin{fact}
If $\dot x$ is a hereditarily symmetric name and $\pi \in \calG$,
then $\pi(\dot x)$ is also hereditarily symmetric.
\end{fact}

For $\bbP$-names $\dot x_0,\dotsc, \dot x_n$, let $\{\dot x_0,\dotsc, \dot x_n\}^\bullet$ be the
$\bbP$-name $\{ \seq{\dot x_0, \mathbbm{1}}, \dotsc, \seq{\dot x_n, \mathbbm{1}}\} $,
and $\seq{\dot x_0, \dot x_1}^{\bullet} =
\{ \{\dot x_0\}^\bullet, \{\dot x_0, \dot x_1\}^\bullet \}^\bullet$.
The name $\{\dot x_0,\dotsc, \dot x_n \}^\bullet$ is a canonical name for the family
of the interpretations of the $\dot x_i$'s,
and $ \seq{\dot x_0, \dot x_1 }^\bullet$ is 
for the ordered pair.
If $\dot x_0,\dotsc, \dot x_n$ are hereditarily symmetric,
then so are $\{\dot x_0,\dotsc, \dot x_n\}^\bullet$ and
$\seq{\dot x_0, \dot x_1}^{\bullet}$.

Let $\mathrm{HS}$ be the class of all hereditarily symmetric names.
For a $(V, \bbP)$-generic $G$, let $\mathrm{HS}^G$ be the class  of all interpretations of
hereditarily symmetric names.
$\mathrm{HS}^G$ is a transitive model of $\ZF$ with
$V \subseteq \mathrm{HS}^G \subseteq V[G]$.
$\mathrm{HS}^G$ is called a \emph{symmetric extension of $V$}.
\\

Here we construct our symmetric extension.
Starting from a model of $\ZFC$,
fix a regular cardinal $\ka$.
In a $\ka$-closed forcing extension of $V$,
we shall construct a symmetric extension 
satisfying the following:
\begin{enumerate}
\item $\AC_\la$ holds for every $\la<\ka$.
\item Every infinite cardinal $\la<\ka$ carries a uniform ultrafilter. 
\item Every ultrafilter over $\ka$ has a measure one set with size $<\ka$.
\item If $0^{\#}$ does not exist,
then for every cardinal $\la > \ka$ and ultrafilter $U$ over $\la$, $U$ has 
a measure one set with size $<\ka$.
\item $\AC_{\ka,2}$ fails.
\end{enumerate}
$0^{\#}$ is a set of natural numbers which code 
true formulas about indiscernibles in the G\"odel's constructible universe $L$.
The existence of $0^{\#}$ is a large cardinal property,
e.g., if $0^{\#}$ exists, 
then every uncountable cardinal is inaccessible in $L$.
In particular $0^{\#}$ does not exist in $L$.
It is also known that
if a measurable cardinal exists, then $0^{\#}$ exists.
See Kanamori \cite{K} for details of $0^{\#}$.

Our symmetric extension is a variant of Feferman's model (see \cite{AC} or \cite{J}),
and essentially the same construction is used in Hayut-Karagila \cite{HK}.

We notice that our symmetric extension may not be a model of Theorems \ref{1.8} and \ref{1.9}.
After constructing symmetric extension, 
we will take an intermediate model between the ground model and a symmetric extension,
which is a required model. This argument is based on Blass's one in \cite{B}.


For a set $X$, let $\mathsf{Fn}(X,2, \mathop{<}\ka)$ be the poset of all
partial functions $p:X \to 2$ with size $<\ka$.
The order is given by the reverse inclusion.
It is a $\ka$-closed poset.
We define a symmetric system $\seq{\bbP, \calG, \calF}$ as follows.
Let $\bbP=\mathsf{Fn}(\ka \times \ka,2, \mathop{<}\ka)$.
Let $\vec{X}=\seq{X_\alpha\mid \alpha<\ka}$ be a sequence of
subsets of $\ka$.
Define the automorphism $\pi_{\vec{X}}$ on $\bbP$ 
by : $\dom(\pi_{\vec{X}}(p))=\dom(p)$, and
\[
\pi_{\vec{X}}(p)(\alpha,\beta)=\begin{cases}
1-p(\alpha,\beta) & \text{if $\beta \in X_\alpha$},\\
p(\alpha,\beta) & \text{if $\beta \notin X_\alpha$}.
\end{cases}
\]
We note that for $\vec{X}$ and $\vec{Y}$,
\[
\pi_{\vec{Y}}(\pi_{\vec{X}}(p))(\alpha,\beta)=\begin{cases}
1-p(\alpha,\beta) & \text{if $\beta \in X_\alpha \triangle Y_\alpha$},\\
p(\alpha,\beta) & \text{if $\beta \notin X_\alpha \triangle Y_\alpha$}.
\end{cases}
\]
Where $X_\alpha \triangle Y_\alpha$ is the symmetric difference of $X_\alpha$ and $Y_\alpha$.
Let $\mathcal G$ be the group of
automorphisms $\sigma$ which is of the form $\pi_{\vec{X}}$ for some sequence $\vec{X}$ of subsets of $\ka$.
For a set $e \in [\ka]^{<\ka}$,
let $\mathrm{fix}(e)=\{\pi_{\vec{X}}\mid X_\alpha=\emptyset$ for every $\alpha \in e\}$.
$\mathrm{fix}(e)$ is a subgroup of $\calG$.
Let $\calF$ be the filter on $\mathcal{G}$ generated by
$\{\mathrm{fix}(e) \mid e \in [\ka]^{<\ka} \}$,
that is, for a subgroup $H \subseteq \calG$,
$H \in \calF \iff \mathrm{fix}(e) \subseteq H$ for some set $e \in [\ka]^{<\ka}$.
It is routine to check that $\calF$ is a normal filter on $\calG$.

For $e \subseteq \ka$ and $p \in \bbP$,
let $p\restriction e $ be the function $p \cap (e \times \ka \times 2)$,
which is in $\mathrm{Fn}(e \times \ka, 2, \mathop{<}\ka)$.

\begin{lemma}\label{4.2}
Let $\dot x_0,\dotsc, \dot x_n$ be  hereditarily symmetric names
and $e \in [\ka]^{<\ka}$ with $\mathrm{fix}(e) \subseteq \bigcap_{i \le n} \mathrm{sym}(\dot x_i)$.
Then for every $p \in \bbP$ and formula $\varphi(v_0,\dotsc, v_n)$ of set theory,
if $p$ forces $\varphi(\dot x_0,\dotsc, \dot x_n)$,
then  $p \restriction e$ also forces $\varphi(\dot x_0,\dotsc, \dot x_n)$.
\end{lemma}
\begin{proof}
Suppose $p \Vdash \varphi(\dot x_0,\dotsc, \dot x_n)$, but
there is $q \le p \restriction e$ with $q \Vdash \neg \varphi(\dot x_0,\dotsc, \dot x_n)$.
Let $d=\{\alpha \in \ka \setminus e \mid \exists \beta (\seq{\alpha,\beta} \in \dom(p))\}$.
Take a sequence $\seq{X_\alpha \mid \alpha<\ka}$ of subsets of $\ka$ 
such that $X_\alpha=\emptyset$ for $\alpha \notin d$,
and for $\alpha \in d$, $X_{\alpha}= \{\beta<\ka: p(\alpha,\beta) \neq q(\alpha,\beta)\}$.
It is easy to check that $\pi_{\vec{X}}(q)$ is compatible with $p$,
$\pi_{\vec{X}}(\dot x_i)=\dot x_i$ for $i \le n$,
but $\pi_{\vec{X}}(q) \Vdash 
\neg \varphi(\pi_{\vec{X}}(\dot x_0),\dotsc, \pi_{\vec{X}}(\dot x_n))$,
this is a contradiction.
\end{proof}

Take a $(V, \bbP)$-generic $G$ and work in $V[G]$.
Let $\mathrm{HS}^G$ be the class of all interpretations of
hereditarily symmetric names.

Since $\bbP$ is $\ka$-closed, we have $([\ka]^{<\ka})^V=([\ka]^{<\ka})^{V[G]}$.
Hence if $\mu<\ka$ is a cardinal and $U \in V$ is an ultrafilter over $\mu$ in $V$,
then $U$ remains an ultrafilter in $V[G]$. Because $V \subseteq \mathrm{HS}^G\subseteq V[G]$,
we have:
\begin{lemma}\label{4.33}
In $\mathrm{HS}^G$, every infinite cardinal $\mu<\ka$ carries a uniform ultrafilter.
\end{lemma}

For $\alpha<\ka$, let 
$G_\alpha=G \cap \mathrm{Fn}(\alpha\times \ka, 2,\mathop{<}\ka)$.
$G_\alpha$ is $(V, \mathrm{Fn}(\alpha\times \ka, 2,\mathop{<}\ka))$-generic.
%
\begin{lemma}\label{4.55}
\begin{enumerate}
\item $V[G_\alpha] \subseteq \mathrm{HS}^G$ for every $\alpha<\ka$.
\item For every ordinal $\gamma$ and $x \in \p(\gamma)^{\mathrm{HS}^G}$,
there is $\alpha<\ka$ with $x \in V[G_\alpha]$.
\end{enumerate}
\end{lemma}
\begin{proof}
(1) $G_\alpha$ can be coded in some element of $\p(\ka)^{V[G_\alpha]}$,
and every element $x$ of $\p(\ka)^{V[G_\alpha]}$ has a
$\mathrm{Fn}(\alpha \times \ka, 2, \mathop{<}\ka)$-name $\dot x$,
which is hereditarily symmetric
with $\fix(\alpha) \subseteq \sym(\dot x)$.
Hence $G_\alpha \in \mathrm{HS}^G$, so $V[G_\alpha] \subseteq \mathrm{HS}^G$.

(2) Take a hereditarily symmetric name $\dot x$ for $x$.
Take $e  \in [\ka]^{<\ka}$
with $\mathrm{fix}(e) \subseteq \mathrm{sym}(\dot x)$,
and take a large $\alpha<\ka$ with $\alpha>\sup(e)$.
By Lemma \ref{4.2}, we have that for every $\xi<\gamma$,
\[
\xi \in x \iff p \Vdash \check \xi \in \dot x \text{ for some $p \in G_\alpha$.}
\]
Hence $x$ is definable in $V[G_\alpha]$, and $x \in V[G_\alpha]$.
\end{proof}
Let $x_\alpha=\{\beta \mid \exists p \in G(p(\alpha,\beta)=1)\}$.
The name $\dot x_\alpha=\{ \seq{ \check \beta, p} \mid
p \in \bbP, p(\alpha, \beta)=1\}$
is  a canonical hereditarily symmetric name for $x_\alpha$
with $\mathrm{fix}(\{\alpha\}) \subseteq \mathrm{sym}(\dot x_\alpha)$,
in particular $x_\alpha \in \mathrm{HS}^G$.
For a sequence $\vec{X}=\seq{X_\alpha\mid \alpha<\ka}$ of
subsets of $\ka$ 
and $\alpha<\ka$, if $\ka \setminus X_\alpha \subseteq \xi$
for some $\xi<\ka$, then $\Vdash \pi_{\vec{X}}(\dot x_\alpha) \cap \dot x_\alpha \subseteq \check \xi$.
Similarly, if $X_\alpha \subseteq \xi$ for some $\xi<\ka$,
then
$\Vdash \pi_{\vec{X}}(\check \ka \setminus \dot x_\alpha) \cap (\check \ka \setminus \dot x_\alpha) \subseteq \check \xi$.

\begin{lemma}\label{4.5}
Let $U \in \mathrm{HS}^{G}$ be an ultrafilter over $\ka$ in $\mathrm{HS}^G$.
Then $U$ has a measure one set with size $<\ka$.
\end{lemma}
\begin{proof}
Take a hereditarily symmetric name $\dot U$ for $U$, and
take $e \in [\ka]^{<\ka}$ with $\mathrm{fix}(e) \subseteq \mathrm{sym}(\dot U)$.
Fix a large $\alpha<\ka$ with $\alpha>\sup(e)$.
If $x_\alpha \in U$,
pick $p \in G$ such that $p \Vdash \dot x_\alpha \in \dot U$.
Take a large $\xi<\ka$
such that $\{\beta \mid \seq{\alpha, \beta} \in \dom(p)\} \subseteq \xi$.
Fix a sequence $\vec{X}=\seq{X_\delta \mid \delta<\ka}$ such that
$X_\delta=\emptyset$ if $\delta \neq \alpha$, 
and $X_\alpha=\ka \setminus \xi$.
Then $\pi_{\vec{X}}(p)=p$
and $p \Vdash$``$\pi_{\vec{X}}(\dot x_\alpha) \in \pi_{\vec{X}}(\dot U)=\dot U$ and
$\pi_{\vec{X}}(\dot x_\alpha) \cap \dot x_\alpha \subseteq \check \xi$.''
Hence $p \Vdash \check \xi \in \dot U$.
The case that $x_\alpha \notin U$ is similar.
\end{proof}

To prove Lemma \ref{4.8} below,
we need the notion of \emph{regular ultrafilters}.
\begin{define}
Let $U$  be an ultrafilter over a set $S$,
and $\la, \mu$  infinite cardinals.
$U$ is said to be \emph{$(\la, \mu)$-regular} if
there is an indexed family $\{X_\alpha \mid \alpha<\mu\} \subseteq U$
such that $\bigcap_{\alpha \in A} X_\alpha =\emptyset$ for every $A \in [\mu]^\la$.
\end{define}
\begin{fact}[Donder \cite{D}, in $\ZFC$]\label{4.7}
Let $\nu$ be a cardinal $ \ge \om_1$.
If there is a uniform ultrafilter over $\nu$ which is 
not $(\om, \mu)$-regular for some cardinal $\mu<\nu$,
then there is an inner model of a measurable cardinal.
In particular $0^{\#}$ exists.
\end{fact}

\begin{lemma}\label{4.8}
Suppose $0^{\#}$ does not exist.
In $\mathrm{HS}^G$, let $\la > \ka$ be an infinite cardinal
and $U$ an ultrafilter over $\la$.
Then $U$ has a  measure one set with size $<\ka$.
\end{lemma}
\begin{proof}
We work in $\mathrm{HS}^G$.
We prove it by induction on $\la \ge \ka$.
The base step $\la=\ka$ is Lemma \ref{4.5}.

Let $\la>\ka$ and $U$ an ultrafilter over $\la$.
We shall prove that 
there is $X \in U$ with $\size{X}<\la$.
Then, by applying the induction hypothesis,
we can take a  measure one set of $U$ with size $<\ka$.

Suppose to the contrary that $\size{X}=\la$ for every $X \in U$,
that is, $U$ is a uniform ultrafilter.

Case 1: There is $f:\la \to \ka$
such that $f^{-1}(Y) \notin U$ for every set $Y \in [\ka]^{<\ka}$.
In this case, consider the ultrafilter $f_*(U)=\{Y \subseteq \ka \mid f^{-1}(Y) \in U\}$.
By the assumption, $f_*(U)$ forms a uniform ultrafilter over $\ka$, 
this contradicts to Lemma \ref{4.5}.

Case 2: For every $f:\la \to \ka$,
there is a set $Y \in [\ka]^{<\ka}$ such that
$f^{-1}(Y) \in U$. 
Take a hereditarily symmetric name $\dot U$ for $U$,
and an ordinal $\alpha<\ka$ with $\mathrm{fix}(\alpha) \subseteq \mathrm{sym}(\dot U)$.
For every $X \in \p(\la)^{V[G_\alpha]}$,
we can find a 
$\mathrm{Fn}(\alpha \times \ka, 2, \mathop{<}\ka)$-name $\dot X$ for $X$.
$\dot X$ is hereditarily symmetric with $\mathrm{fix}(\alpha) \subseteq \mathrm{sym}( \dot X)$.
By Lemma  \ref{4.2}, we have that for every $X \in \p(\la)^{V[G_\alpha]}$,
\[
X \in U \iff p \Vdash \dot X \in \dot U \text{ for some $p \in G_\alpha$}.
\]
Thus we have $U \cap V[G_\alpha] \in V[G_\alpha]$.
Let $U'=U \cap V[G_\alpha]$,
which is a uniform ultrafilter over $\la$ in $V[G_\alpha]$.

Subcase 1: $\ka=\om$.
In this case, it is easy to check that $U'$ is  $\sigma$-complete in $V[G_\alpha]$,
hence $U'$ is a $\sigma$-complete uniform ultrafilter over $\la$.
This means, in $V[G_\alpha]$, there is a measurable cardinal $\le \la$,
hence $0^{\#}$ exists in $V[G_\alpha]$.
This contradicts to the assumption.

Subcase 2: $\ka>\om$.
We prove that $U'$ is not $(\om, \ka)$-regular in $V[G_\alpha]$.
Since $V[G_\alpha]$ is a model of $\ZFC$, we conclude that $0^{\#}$ exists by Fact \ref{4.8},
this is also impossible.

We work in $V[G_\alpha]$. Suppose to the contrary that $U'$ is $(\om, \ka)$-regular, and 
take $\{X_\alpha\mid \alpha<\ka \} \subseteq U'$
such that $\bigcap_{\alpha \in A} X_\alpha=\emptyset$ for every $A \in [\ka]^\om$.
Define $g:\la \to \ka$ as follows. For $\beta<\la$,
by the choice of the $A_\alpha$'s, the set $\{\alpha<\ka \mid \beta \in A_\alpha\}$ is bounded in $\ka$.
Let $g(\beta)=\sup\{\alpha <\ka \mid \beta \in A_\alpha\}<\ka$.
By the assumption, in $\mathrm{HS}^G$, there is a set $Y \in [\ka]^{<\ka}$
such that $g^{-1}(Y) \in U$.
We know $Y \in V[G_\alpha]$ since $Y$ has cardinality $<\ka$,
so $g^{-1}(Y) \in V[G_\alpha]$ and $g^{-1}(Y) \in U'$. 
Pick $\gamma <\ka$ with $\gamma>\sup(Y)$.
By the definition of $g$, we know that $\beta \notin A_\gamma$ for every $\beta \in g^{-1}(Y)$,
hence $A_\gamma \cap g^{-1}(Y)=\emptyset$.
This is a contradiction.
\end{proof}


To check that $\AC_{\ka,2}$ fails in $\mathrm{HS}^G$,
we need more.
For $\alpha<\ka$,
the name $\dot x_\alpha=\{\seq{\check \beta, p} \mid p \in \bbP, p(\alpha,\beta)=1\}$ 
is a canonical hereditarily symmetric name
for $\dot x_\alpha$ with $\fix(\{\alpha\}) \subseteq \sym(\dot x_\alpha)$.
Then for $\pi_{\vec{X}} \in \mathcal G$,
the name $\pi_{\vec{X}}(\dot x_\alpha)$ is a hereditarily symmetric name
for the symmetric difference $x_\alpha \mathbin \triangle X_\alpha$ with
$\fix(\{\alpha\}) \subseteq \sym(\pi_{\vec{X}}(\dot x_\alpha))$.

Now fix a uniform ultrafilter $W$ over $\ka$ in $V$.
Let $d_\alpha$ be the set $\{x_\alpha \mathbin \triangle Y \mid Y \in \p(\ka)^V, Y \notin W\}$,
and $\dot d_\alpha$ the name $\{\seq{\pi_{\vec{X}}(\dot x_\alpha), \mathbbm 1} \mid \pi_{\vec{X}} \in \calG,
X_\alpha \notin W\}$.
$\dot d_\alpha$ is a hereditarily symmetric name for $d_\alpha$.
Similarly, let $e_\alpha$ be the set 
$\{x_\alpha \mathbin \triangle Y \mid Y \in \p(\ka)^V, Y \in W\}$,
and $\dot e_\alpha$ the name $\{\seq{\pi_{\vec{X}}(\dot x_\alpha), \mathbbm 1} \mid \pi_{\vec{X}} \in \calG,
X_\alpha \in W\}$.
$\dot e_\alpha$ is a hereditarily symmetric name
for $e_\alpha$.
We note that $d_\alpha, e_\alpha \in V[G_{\alpha+1}]$ and $d_\alpha \cap e_\alpha=\emptyset$.

The following is straightforward:
\begin{lemma}\label{5.10}
\begin{enumerate}
\item $\pi_{\vec{X}}(\dot d_{\alpha})=
\begin{cases}
\dot d_{\alpha} & \text{if $X_\alpha \notin W$,}\\
\dot e_{\alpha} & \text{if $X_\alpha \in W$}.
\end{cases}$
\item $\pi_{\vec{X}}(\dot e_{\alpha})=
\begin{cases}
\dot e_{\alpha} & \text{if $X_\alpha \notin W$,}\\
\dot d_{\alpha} & \text{if $X_\alpha \in W$}.
\end{cases}$
\end{enumerate}
\end{lemma}

Define the function $F_0$ on $\ka$ by
$F_0(\alpha)=\{d_\alpha, e_\alpha\}$ for $\alpha<\ka$.
\begin{lemma}
$F_0 \in \mathrm{HS}^G$.
\end{lemma}
\begin{proof}
By Lemma \ref{5.10}, the name $\{\dot d_\alpha,\dot e_\alpha\}^\bullet$
is hereditarily symmetric with $\sym(\{\dot d_\alpha,\dot e_\alpha\}^\bullet)=\calG$.
Hence the name $\dot F_0=\{\seq{\seq{\check \alpha, \{\dot d_\alpha,\dot e_\alpha\}^\bullet }^\bullet, \mathbbm 1} \mid \alpha<\ka\}$, which is a name for $F_0$,
is hereditarily symmetric.
So we have $F_0 \in \mathrm{HS}^G$.
\end{proof}

\begin{lemma}
$\mathsf{AC}_{\ka, 2}$ fails in $\mathrm{HS}^G$.
\end{lemma}
\begin{proof}
Let $E_\alpha=\{d_\alpha ,e_\alpha\}$ for $\alpha<\ka$,
and consider the indexed family $\{E_\alpha \mid \alpha<\ka\}$,
which is in $\mathrm{HS}^G$ because $F_0 \in \mathrm{HS}^G$.
We show  that this family has no choice function in $\mathrm{HS}^G$.

Suppose to the contrary that the family has a choice function.
Take a choice function $f \in \mathrm{HS}^G$,
and let $\dot f$ be a hereditarily symmetric name for $f$.
Take $p \in G$ and $\alpha<\ka$ such that $\fix(\alpha) \subseteq \sym(\dot f)$ and
$p$ decides the value of  $\dot f(\check \alpha)$,
say $p \Vdash \dot f (\check \alpha)=\dot d_\alpha$.
Fix a large $\beta<\ka$
with $\dom(p) \cap (\{\alpha\} \times \ka) \subseteq \{\alpha\} \times \beta$.
In $V$, take a sequence $\vec{X}$ such that
$X_\gamma=\emptyset$ if $\gamma \neq \alpha$, and
$X_\alpha=\ka \setminus \beta$. 
We have $\pi_{\vec{X}}(p)=p$,
and since $X_\alpha \in W$, we have $p \Vdash
\dot f(\check \alpha)=\pi_{\vec{X}}(\dot f)(\check \alpha)=
\pi_{\vec{X}}(\dot d_\alpha)=\dot e_\alpha$
by Lemma \ref{5.10}.
Hence $p \Vdash \dot d_\alpha=\dot e_\alpha$,
this is a contradiction.
\end{proof}

As mentioned before, our symmetric extension $\mathrm{HS}^G$ may not be a model of Theorem \ref{1.8} and \ref{1.9},
because we do not know if $\LT$ holds in $\mathrm{HS}^G$.
Instead, we take an intermediate model $M$ between $V$ and $\mathrm{HS}^G$
such that $M$ satisfies not only (1)--(5) but also the following (6) and (7):
\begin{enumerate}
\item[(6)] Every ultrafilter has a well-orderable measure one set.
\item[(7)] If $0^{\#}$ does not exist,
then every ultrafilter has a well-orderable measure one set with size $<\ka$.
\end{enumerate}

For $\alpha<\ka$, if $\dot x$ is a $\bbP$-name which is of the form
$\{ \seq{ \check \beta, p} \mid \beta<\ka, p \in A_\beta \}$ 
for some $A_\beta \subseteq \mathrm{Fn}(\alpha\times \ka, 2, <\ka)$,
then $\dot x$ is a hereditarily symmetric name with $\mathrm{fix}(\alpha) \subseteq \mathrm{sym}(\dot x)$.
Let $Q_\alpha$ be the set of such names,
and $\dot P_\alpha=\{\seq{\dot x, \mathbbm 1} \mid \dot x \in Q_\alpha\}$.
$\dot P_\alpha$ is a hereditarily symmetric name for $\p(\ka)^{V[G_\alpha]}$ with
$\mathrm{sym}(\dot P_\alpha)=\calG$.
Using this, we have:
\begin{lemma}
The function $F_1$ on $\ka$ defined by
$F_1(\alpha)=\p(\ka)^{V[G_\alpha]}$ is in $\mathrm{HS}^G$.
\end{lemma}


Our final model $M$ is $V(\{F_0, F_1\})$,
which is the minimal inner model of $\ZF$ containing 
$F_0, F_1$ and all elements of $V$.
$M$ is an intermediate model between $V$ and $\mathrm{HS}^G$.
We note that $V[G_\alpha] \subseteq M$ for every $\alpha<\ka$
since $G_\alpha \in M$.
We show that $M$ satisfies the properties (1)--(5), and (6), (7).

For every ordinal $\gamma$, we have
$\p(\gamma)^M=\p(\gamma)^{\mathrm{HS}^G}$ by Lemma \ref{4.55}.
By Lemmas \ref{4.33}, \ref{4.8}, and this observation, we have:
\begin{lemma}\label{4.11}
The following hold in $M$:
\begin{enumerate}
\item For every infinite cardinal $\la<\ka$, $\la$ carries a uniform ultrafilter. 
\item Every ultrafilter over $\ka$ has a measure one set with size $<\ka$.
\item If $0^{\#}$ does not exist, then for every cardinal $\la > \ka$ and ultrafilter $U$ over $\la$, $U$ has 
a measure one set with size $<\ka$.
\end{enumerate}
\end{lemma}

Since $F_0 \in M \subseteq \mathrm{HS}^G$, we have:
\begin{lemma}\label{5.16}
$\AC_{\ka,2}$ fails in $M$.
\end{lemma}

Define the function $H$ on $\ka$ by
$H(\alpha)=F_1(\alpha) \cup \bigcup_{\beta<\alpha} F_0(\beta)=\p(\ka)^{V[G_\alpha]} \cup \{d_\beta, e_\beta \mid \beta <\alpha\}$.
Note that $H \in M$, $H(\alpha) \in V[G_{\alpha}]$, 
and whenever $\alpha<\beta<\ka$ we have $H(\alpha) \subseteq H(\beta)$.

\begin{lemma}\label{5.16+}
In $M$,
for every set $S \in M$,
there is an ordinal $\gamma$ and a surjection
from $\gamma \times {}^{<\om}(\bigcup_{\alpha<\ka} H(\alpha))$ onto $S$.
\end{lemma}
\begin{proof}

First we show that, in $M$, there is a surjection from $\ka \times \bigcup_{\alpha<\ka} H(\alpha)$
onto $\mathrm{trcl}(\{F_0, F_1\})$.
Since $\mathrm{trcl}(\{F_0, F_1\})=\{F_0, F_1\} \cup
\mathrm{trcl}(F_0) \cup \mathrm{trcl}(F_1)$,
it is sufficient to show that there is a surjection from
$\ka \times \om \times \bigcup_{\alpha<\ka} H(\alpha)$ onto $\mathrm{trcl}(F_0)$ and from
$\ka \times \om\times \bigcup_{\alpha<\ka} H(\alpha)$ onto $\mathrm{trcl}(F_1)$.
We show only for $\mathrm{trcl}(F_0)$,
we can take a surjection onto $\mathrm{trcl}(F_1)$ by a similar way.

Recall that for a set $X$, the transitive closure $\mathrm{trcl}(X)$ is defined as follows:
$X_0=X$, $X_{n+1}=\bigcup X_n$, and $\mathrm{trcl}(X)=\bigcup_n X_n$.
Hence the set $\mathrm{trcl}(F_0)$ is the union:
\begin{align*}
\{\seq{\alpha, \{d_\alpha, e_\alpha\}} \mid \alpha<\ka\}
&\cup \{\{\alpha\}, \{\alpha, \{d_\alpha, e_\alpha\}\} \mid \alpha<\ka\}\\
&\cup \{\alpha \mid \alpha<\ka\} \cup \{\{d_\alpha, e_\alpha\} \mid \alpha<\ka\}\\
& \cup \{d_\alpha, e_\alpha \mid \alpha<\ka\} \cup \bigcup_{\alpha<\ka} (d_\alpha \cup e_\alpha).
\end{align*}

Define $f:\ka \times \om \times \bigcup_{\alpha<\ka} H(\alpha) \to \mathrm{trcl}(F_0)$ in $M$
as follows.
\begin{itemize}
\item $f(\alpha,0,x)=\seq{\alpha, \{d_\alpha, e_\alpha\}}$.
\item $f(\alpha,1, x)=\{\alpha\}$ and 
$f(\alpha,2, x)=\{\alpha, \{d_\alpha, e_\alpha\}\} $.
\item $f(\alpha,3,x)=\alpha$
and $f(\alpha,4,x)=\{d_\alpha, e_\alpha\}$.
\item $f(\alpha, n, x)=x$ for $n \ge 5$.
\end{itemize}
It is straightforward to check that this $f$ works.

Now fix a set $S \in M$.
By Fact \ref{model},
there is a function $f' \in M$ and an ordinal $\gamma'$ such that
$f'$ is a surjection from ${}^{<\om} \gamma' \times {}^{<\om}(\mathrm{trcl}(\{F_0, F_1\}))$
onto $S$.
We have known that there is a surjection from $\ka \times \bigcup_{\alpha<\ka} H(\alpha)$
onto $\mathrm{trcl}(\{F_0, F_1\})$.
Hence, by modifying $f'$ in $M$, 
we can take a map $f$ and an ordinal $\gamma$ such that
$f$ is a surjection from $\gamma \times {}^{<\om}
(\bigcup_{\alpha<\ka} H(\alpha))$
onto  $S$.
\end{proof}

We now check that $M$ satisfies the condition (1).
\begin{lemma}\label{5.19}
In $V[G]$, $M$ is closed under $<\ka$-sequences.
In particular $\AC_\la$ holds for every $\la<\ka$ in $M$.
\end{lemma}
\begin{proof}
Take $\la<\ka$ and a set $\{y_\alpha \mid \alpha<\la\} \subseteq M$.
In $M$, we can take a large set $S$, a map $f$, and an ordinal $\gamma$
such that $\{y_\alpha \mid \alpha<\ka\} \subseteq S$ and 
$f$ is a surjection from $\gamma \times^{<\om} (\bigcup_{\alpha<\ka} H(\alpha))$ onto $S$.
Then there is $\delta<\ka$
such that $\{y_\alpha \mid \alpha<\la\} \subseteq f``(\gamma \times {}^{<\om} H(\delta))$.
In $M$, $H(\delta)$ is well-orderable;
$H(\delta) \in V[G_{\delta}] \subseteq M$ and
$H(\delta)$ is well-orderable in $V[G_{\delta}]$.
In particular $\gamma \times {}^{<\om} H(\delta)$ is also well-orderable in $M$.
So, by modifying $f$ in $M$,
we can take an ordinal $\theta$ and a map $g$ such that $\dom(g)=\theta$
and $\{y_\alpha \mid  \alpha<\la \} \subseteq \mathrm{range}(g)$.
In $V[G]$, pick $\xi_\alpha<\theta$ with $g(\xi_\alpha)=y_\alpha$.
Then $\seq{\xi_\alpha \mid \alpha<\la} \in V \subseteq M$,
hence $\{y_\alpha \mid \alpha<\mu \} =g``\{\xi_\alpha \mid \alpha<\la\} \in M$.
\end{proof}
Next we show that $M$ satisfies the conditions (6) and (7).

\begin{lemma}\label{4.12}
In $M$, let $S$ be an infinite set, and 
$U$ an ultrafilter over $S$.
Then $U$ has  a well-orderable  measure one set.
In addition, if $0^{\#}$ does not exist,
then $U$ has a well-orderable measure one set  with size $<\ka$.
\end{lemma}
\begin{proof}
We work in $M$.
Take a map $f$ and an ordinal $\gamma$ such that
$f$ is a surjection from $\gamma \times {}^{<\om}(\bigcup_{\alpha<\ka} H(\alpha))$
onto $S$.

For $\alpha<\ka$, let $S_\alpha=f``(\gamma \times {}^{<\om} H(\alpha))$.
We know $S=\bigcup_{\alpha<\ka} S_\alpha$,
and $\seq{S_\alpha \mid \alpha<\ka} \in M$.
Define $g:S \to \ka$ as
$g(x)$ is the least $\alpha<\ka$ with $x \in S_\alpha$.
Consider the ultrafilter $g_*(U)$ over $\ka$.
We can find an ordinal $\alpha<\ka$ with  $\alpha \in g_*(U)$ by Lemma \ref{4.11},
so we have $g^{-1}(\alpha) \subseteq S_\alpha=f``(\gamma \times {}^{<\om} H(\alpha)) \in U$.
Hence, we may assume that $U$ is an ultrafilter over $S_\alpha$.
In $M$, $H(\alpha)$ is well-orderable.
Hence, in $M$, $S_\alpha=f``(\gamma \times {}^{<\om} H(\alpha))$ is a well-orderable measure one set of $U$.
Moreover, if $0^{\#}$ does not exist, 
then by Lemma \ref{4.8}, $U$ has a well-orderable measure one set with size $<\ka$.
\end{proof}

We are ready to prove Theorem \ref{1.8}.
We use the well-known fact that if the Countable Choice holds then
every infinite set has a countably infinite subset.
\begin{thm}\label{4.1}
If $\ZFC$ is consistent, then so  is $\ZF+\LT+\mathsf{WUF}+\neg \mathsf{AC}_{\om_1,2}$.
\end{thm}
\begin{proof}
Assuming $\AC+$``$0^{\#}$ does not exist'' (e.g., suppose $V=L$), take the model $M$ as the above with the case $\ka=\om_1$.
The following hold in $M$:
\begin{enumerate}
\item The Countable Choice holds (Lemma \ref{5.19}).
\item $\om$ carries a non-principal ultrafilter (Lemma \ref{4.11}).
\item Every ultrafilter has a countable measure one set (Lemma \ref{4.12}).
\item $\AC_{\om_1,2}$ fails (Lemma \ref{5.16}).
\end{enumerate}
By (1) and (3) with Lemma \ref{2.1+}, $\LT$ holds.
By (1), every infinite set $S$ has a countably infinite subset $S'$,  and 
(2) guarantees that there is a non-principal ultrafilter $U'$ over $S'$.
Then we can extend $U'$ to a non-principal ultrafilter $U$ over $S$.
\end{proof}

The resulting model of Theorem \ref{4.1} has 
non-principal ultrafilters, but
no uniform ultrafilter over uncountable cardinals.

\begin{question}\label{5.18}
Is $\ZF+\LT+\neg \mathsf{AC}+$``every infinite cardinal carries a uniform ultrafilter''
relatively consistent with $\ZFC$?
\end{question}
We do know the answer to this question,
but we prove  that the statement $\LT+$``every infinite cardinal carries a uniform ultrafilter''
deduces more fragments of $\mathsf{AC}$ than the Countable Choice.

\begin{prop}
Suppose $\LT$ holds.
Let $\ka$ be an infinite cardinal.
If every infinite cardinal $\la \le \ka$ carries a uniform ultrafilter,
then $\mathsf{AC}_\la$ holds for every $\la \le \ka$.
\end{prop}
\begin{proof}
By induction on $\la \le \ka$.
The base step $\la=\om$ follows from Theorem \ref{2.1}.
If $\la$ is singular,
then $\mathsf{AC}_\la$ holds by the induction hypothesis as follows:
Take a family $\{A_\alpha \mid \alpha<\la\}$ of non-empty sets.
Take also an increasing sequence $\seq{\la_i \mid i<\cf(\la)}$ with limit $\la$.
For each $i<\cf(\la)$,
by the induction hypothesis,
there is a choice function of $\{A_\alpha \mid \alpha<\la_i\}$.
Again, by $\mathsf{AC}_{\cf(\la)}$,
we can choose a family $\{f_i \mid i<\cf(\la)\}$ such that 
each $f_i$ is a choice function of $\{A_\alpha \mid \alpha<\la_i\}$.
Then the function $f$ on $\la$ defined by
$f(\alpha)=f_i(\alpha)$, where $i$ is the least $i<\cf(\la)$ with $\alpha<\la_i$,
is a choice function of $\{A_\alpha \mid \alpha<\la\}$.

Now suppose $\la$ is regular. Fix a uniform ultrafilter $U$ over $\la$.
Let $\{B_\alpha \mid \alpha<\la\}$ be a family of non-empty sets.

For $\alpha<\la$, let $A_\alpha=\prod_{\beta \le \alpha} B_\beta$.
By the induction hypothesis, $A_\alpha$ is non-empty.
By $\LT$, there is a function $f$ on $\la$
such that $X=\{\alpha<\ka \mid f(\alpha) \in A_\alpha\} \in U$.
Since $U$ is uniform, we can define the function $g$ on $\la$ as follows:
$g(\alpha)=f(\beta)(\alpha)$, where $\beta $ is the least element of $X \setminus (\alpha+1)$.
It is easy to check that
$g(\alpha) \in B_\alpha$ for every $\alpha<\la$.
\end{proof}

It is known that if $\AC_\ka$ holds for every cardinal $\ka$,
then the Dependent Choice $\mathsf{DC}$ holds (see Jech \cite{J}).
\begin{prop}
Suppose $\LT$. If every infinite cardinal carries a uniform ultrafilter,
then $\mathsf{AC}_\ka$ holds for every infinite cardinal $\ka$,
and $\mathsf{DC}$ holds.
\end{prop}

\begin{question}
Suppose $\LT$.
Does $\mathsf{DC}$ hold provided 
that $\om$ carries a non-principal ultrafilter?
\end{question}

\section{Choiceless models of $\LT$ (II)}\label{sec6}
By using a model $M$ in Section \ref{sec5},
we will give a proof of Theorem \ref{1.9}.
For this sake, we need the following lemma and theorem.

\begin{lemma}\label{7.1}
The following are  equivalent:
\begin{enumerate}
\item There is an ultrafilter which is not $\sigma$-complete.
\item $\om$ carries a non-principal ultrafilter.
\end{enumerate}
\end{lemma}
\begin{proof}
(2) $\Rightarrow$ (1) is trivial.

For (1) $\Rightarrow$ (2),
fix an ultrafilter $U$ over a set $S$
which is not $\sigma$-complete.
We can find a map $g:S \to \om$ such that
$g^{-1}(n) \notin U$ for every $n<\om$.
Then $g_*(U)$ is a non-principal ultrafilter over $\om$.
\end{proof}

Here we recall some basic definitions for filters.
Let $S$ be a set and $F$ a filter over $S$.
A set $X \subseteq S$ is \emph{$F$-positive} if $X \cap Y \neq \emptyset$ for every $Y \in F$.
For an $F$-positive set $X$,
let $F \restriction X=\{Y \subseteq S \mid Y \cup (S \setminus X) \in F\}$.
$F \restriction X$ is a filter over $S$ with $X \in F \restriction X$.
Let $\ka$ be a cardinal. A filter $F$ is \emph{$\ka$-saturated} 
if for every indexed family $\{X_\alpha \mid \alpha<\ka\}$ of positive sets,
there are $\alpha<\beta<\ka$ such that $X_\alpha \cap X_\beta$ is $F$-positive.

See Kanamori \cite{K} for the proof of the following Tarski's theorem.
\begin{fact}[Tarski, in $\ZFC$]\label{Tarski}
If $F$ is a $(2^\om)^+$-complete $\om_1$-saturated filter,
then there is an $F$-positive set $X$ such that
$F \restriction X$ is an ultrafilter.
\end{fact}

\begin{thm}\label{4.19}
Suppose $V$ satisfies $\mathsf{AC}$.
Then there is a c.c.c.\ forcing extension $V[G]$
and an intermediate model $M$ between $V$ and $V[G]$
satisfying the following:
\begin{enumerate}
\item $M$ is a model of $\ZF+\LT$.
\item If $S \in V$ is a set, $\ka$ is an uncountable cardinal, and $U\in V$ is a $\ka$-complete ultrafilter over $S$ in $V$,
then the filter generated by $U$, $\{X\in \p(S)^M \mid \exists Y\in U(Y \subseteq X)\}$, is a $\ka$-complete ultrafilter in $M$.
\item In $M$, every ultrafilter is $\sigma$-complete.
\item $\AC_{\om,2}$ fails in $M$.
\end{enumerate}
\end{thm}
\begin{proof}
Our model $M$ is given in Section \ref{sec5} with the case $\ka=\om$.
Note that $\mathrm{Fn}(\om \times \om, 2, <\om)$ has the c.c.c.

We already know that the model $M$ satisfies the following:
\begin{enumerate}
\item There is no non-principal ultrafilter over $\om$.
\item $\mathsf{AC}_{\om,2}$ fails.
\item Every ultrafilter has a well-orderable measure one set.
\end{enumerate}
By (1) and Lemma \ref{7.1}, we also have:
\begin{enumerate}
\item[(4)] Every ultrafilter is $\sigma$-complete.
\end{enumerate}

\begin{claim}
If $U \in V$ is a $\ka$-complete ultrafilter in $V$ for some uncountable cardinal $\ka$,
then $U$ generates a $\ka$-complete ultrafilter in $M$.
\end{claim}
\begin{proof}
Since $\bbP$ is countable, it is known that $U$ generates a $\ka$-complete ultrafilter in $V[G]$
(e.g., see Kanamori \cite{K}).
Hence it also generates a $\ka$-complete ultrafilter in $M$.
\end{proof}

\begin{claim}\label{6.11}
In $M$, let $\ka$ be an infinite cardinal,
and $U$ an ultrafilter over $\ka$.
Then $U \cap V \in V$.
\end{claim}
\begin{proof}
Let $\la=(2^\om)^V$. First we prove that
$U$ is $\la^+$-complete.
If not, then we can find a cardinal $\mu \le \la$
and $f:\ka \to \mu$ 
such that
$f^{-1}(\gamma) \notin U$ for every $\gamma<\mu$.
Take $\mu$-many reals $\{r_\xi \mid \xi<\mu\} \subseteq \p(\om)^V$.
Define $X_n \in U$ as follows.
For $n<\om$, we know $\{\alpha \mid n \in r_{f(\alpha)}\}\in U$ or
$\{\alpha \mid n \notin r_{f(\alpha)}\}\in U$.
If 
$\{\alpha \mid n \in r_{f(\alpha)}\}\in U$ then
let $X_n=\{\alpha \mid n \in r_{f(\alpha)}\}$,
otherwise $X_n$ is $\{\alpha \mid n \notin r_{f(\alpha)}\}$.
Since $U$ is $\sigma$-complete,
we have $X=\bigcap_{n<\om} X_n \in U$.
By the assumption, we can pick $\alpha, \beta\in X$
 with $f(\alpha) \neq f(\beta)$.
Then for every $n<\om$, we have $n \in r_{f(\alpha)} \iff n \in r_{f(\beta)}$,
hence $r_{f(\alpha)}=r_{f(\beta)}$.
This is a contradiction.

To show that $U \cap V \in V$,
fix a name $\dot U$ for $U$ and $p \in \bbP$ such that
$p \Vdash$``$\dot U$ is an ultrafilter over $\ka$ in $M$.''
In $V$, let $F=\{X \subseteq \ka \mid p \Vdash \check X \in \dot U\}$.
Since $U$ is a $\la^+$-complete ultrafilter and $\bbP$ satisfies the c.c.c.,
we have that, in $V$, $F$ is a $(2^\om)^+$-complete $\om_1$-saturated 
filter over $\ka$.
By Fact \ref{Tarski} there must be an $F$-positive set $X \subseteq \ka$
 such that $F \restriction X$
is an ultrafilter. 
Pick $q \le p$ with $q \Vdash \check X \in \dot U$.
Then $q \Vdash \check{F} \restriction \check X=\dot U \cap V$,
as required.
\end{proof}

The next claim together with 
Lemma \ref{2.1+} yields $\LT$ in $M$.
\begin{claim}\label{6.6}
In $M$, let $U$ be an ultrafilter over $S$.
Then for every indexed family $\{A_s \mid s \in S\}$ of non-empty sets,
there is a function $g$ on $S$ such that
$\{s \mid g(x) \in A_s\} \in U$.
\end{claim}
\begin{proof}
Since $U$ has a well-orderable measure one set,
we may assume that $U$ is a non-principal $\sigma$-complete ultrafilter over an infinite cardinal $\ka$.
Let $U'=U \cap V$, which is a $\sigma$-complete ultrafilter over $\ka$ in $V$ by Claim \ref{6.11}.
In $V$, let $j:V \to N\approx {}^I V/ U'$ be the ultrapower elementary embedding 
induced by $U'$ where we identify the ultrapower ${}^I V/U'$ with its transitive collapse $N$.
Since $\bbP$ is countable,
in $V[G]$, we can extend $j$ to $j:V[G] \to N[G]$ by the canonical way.
Moreover we have that $j(F_0)=F_0$, $j(F_1)=F_1$, and 
$j \restriction V(\{F_0, F_1\})$ is an elementary embedding from
$V(\{F_0,F_1\}) $ to $N(\{F_0, F_1\})$.
We also note $j(\bigcup_n H(n))=\bigcup_n H(n)$.
Let $\mathrm{id}$ be the identity map on $\ka$.
We know that $U=\{X \in \p(\ka)^M \mid [\mathrm{id}] \in j(X)\}$,
and for every ordinal $\gamma$,
there is a function $f$ on $\ka$ in $V$ such that
$j(f)([\mathrm{id}])=\gamma$.

Now fix an indexed family $\{A_\alpha \mid \alpha<\ka\}\in M$ of non-empty sets.
Let $Y=\bigcup_{\alpha<\ka} A_\alpha$,
and take an ordinal $\gamma$ and a map $g \in M$
such that $g$ is a surjection from $\gamma \times {}^{<\om}(\bigcup_{n<\om} H(n))$ onto $Y$.
By the elementarity,
$j(g)$ is a surjection from $j(\gamma) \times {}^{<\om}(\bigcup_{n<\om} H(n))$ onto
$j(Y)$ and $j(\bigcup_{\alpha<\ka} A_\alpha) = j(Y)$.
Consider the family $j(\{A_\alpha \mid \alpha<\ka\})=\{A^*_\alpha \mid \alpha<j(\ka)\}$.
We know that $A^*_{[\mathrm{id}]}$ is non-empty, so we can take $\delta <j(\gamma)$ and
$\seq{r_0,\dotsc, r_k} \in {}^{<\om}(\bigcup_n H(n))$ 
such that $j(g)(\delta, \seq{r_0,\dotsc, r_k}) \in A^*_{[\mathrm{id}]}$.
Pick a function $h \in V$ with $j(h)([\mathrm{id}])=\delta$.
By the elementarity of $j$ again,
we have that
$\{\alpha<\ka \mid g(h(\alpha), \seq{r_0,\dotsc, r_k}) \in A_\alpha\}\in U$.
Hence $\alpha \mapsto g(h(\alpha), \seq{r_0,\dotsc, r_k})$ is a required
choice function in $M$.
\end{proof}

\end{proof}

Now we have Theorem \ref{1.9}.

\begin{thm}\label{6.13}
If $\ZFC+$``there is a measurable cardinal'' is 
consistent, then so is
$\ZF+\LT+\neg \mathsf{AC}_{\om,2}+$``there is a measurable cardinal''$+$
``every ultrafilter is $\sigma$-complete.''
\end{thm}
\begin{proof}
Suppose $V$ satisfies $\mathsf{AC}$ and there is a measurable cardinal $\ka$.
In the resulting model $M$ in Theorem \ref{4.19},
$\ka$ remains a measurable cardinal.
\end{proof}

If we suppose the existence of a strongly compact cardinal,
we can give a partial answer to Question \ref{5.18}.
Here we recall strongly compact  cardinals.
An uncountable cardinal $\ka$ is a \emph{strongly compact cardinal}
if for every cardinal $\la \ge \ka$,
there is a $\ka$-complete fine ultrafilter $U$ over $[\la]^{<\ka}$,
that is, $\{x \in [\la]^{<\ka} \mid \alpha \in x\} \in U$ for every $\alpha<\la$.
Every strongly compact cardinal is measurable.
If $\ka$ is strongly compact,
then for every regular $\la \ge \ka$,
there is a $\ka$-complete uniform ultrafilter over $\la$;
Fix a $\ka$-complete fine ultrafilter over $\la$.
Define $f:[\la]^{<\ka}\to \la$ by $f(x)=\sup(x)$.
Then one can check that $f_*(U)$ is a $\ka$-complete uniform ultrafilter over $\la$.
In $\ZFC$, the converse of this fact hold:
$\ka$ is strongly compact if  every regular $\la \ge \ka$ carries a
$\ka$-complete uniform ultrafilter.
We do not know if the converse still holds in $\ZF$.

\begin{thm}
If $\ZFC+$``there is a strongly compact cardinal'' is 
consistent, then so is
$\ZF+\LT+\neg \mathsf{AC}_{\om,2}+$
``every ultrafilter is $\sigma$-complete''+``there is a measurable cardinal $\ka$ such that
every regular cardinal $\la \ge \ka$
carries a uniform ultrafilter.''
\end{thm}
\begin{proof}
Suppose $V$ satisfies $\mathsf{AC}$ and there is a strongly compact cardinal $\ka$.
In $V$, for every regular $\la \ge \ka$, there is a $\ka$-complete uniform ultrafilter over $\la$.
Then, in the resulting model $M$ in Theorem \ref{4.19},
every regular cardinal $\la \ge \ka$ carries a $\ka$-complete uniform ultrafilter.
\end{proof}

\section{Consistency strengths of some statements}\label{sec7}
We proved that $\LT+\mathsf{WUF(?)}$ does not imply 
$\AC_{\om, 2}$.
However our proof needed a large cardinal assumption.
In this section, we show that 
this large cardinal assumption is necessary.

\begin{prop}\label{6.32}
Suppose $\LT$.
Let $U$ be a non-principal ultrafilter over a set $S$.
Then there is a cardinal $\ka$ such that
$U$ is not $\ka$-complete.
\end{prop}
Note that, in $\ZF$, it is possible that
there is  a non-principal ultrafilter which is $\ka$-complete for every cardinal $\ka$.
We will return to this topic in the end of this section.

\begin{proof}
Suppose to the contrary that $U$ is $\ka$-complete for every cardinal $\ka$.
First we shall prove:
\begin{claim}
For every ordinal $\alpha$ and $f:S \to \alpha$,
there is $\beta<\alpha$ with $f^{-1}(\{\beta\}) \in U$.
\end{claim}
\begin{proof}
Suppose not. For $\beta<\alpha$, let $S_\beta=\{x \in S \mid f(x) \neq \beta \}$,
which is in $U$.
By the assumption,
we know $\bigcap_{\beta<\alpha} S_\beta \in U$,
in particular, there is $x \in \bigcap_{\beta<\alpha} S_\beta \in U$.
Then $f(x) \neq \beta$ for every $\beta<\alpha$, this is impossible.
\end{proof}

Next, fix a large limit ordinal $\theta$ such that $V_\theta$ is a model of a sufficiently large
fragment of $\ZF$.
Let $M$ be the structure $\seq{V_\theta; \in}$,
and $M^*=\seq{M^*; E}$ the ultrapower ${}^S M/U$.
By $\LT$, $M^*$ is a model of (a large fragment of) $\ZF$.
Let $j:V_\alpha \to M^*$ be the ultrapower embedding,
that is, $j(a)=[c_a]$ for $a \in V_\theta$, where $c_a$ is the constant function with value $a$.
By $\LT$, the ultrapower embedding $j$ is an elementary embedding.
By the claim, for every $[f] \in M^*$, if $[f]$ is an ordinal of $M^*$, 
then $[f]=j(\alpha)$ for some ordinal $\alpha$.
Thus, we have that the set of ordinals of $M^*$ is isomorphic to some set of ordinals.
Hence the relation $E$ is well-founded on the ordinals of $M^*$.
Because $M^*$ is a model of (a large fragment of) $\ZF$,
every set in $M^*$ has a rank in $M^*$. Thus we conclude that
$E$ is a well-founded relation on $M^*$,
and we can take the transitive collapse 
$N^*$ of $M^*$ and the collapsing map $\pi:M^* \to N^*$.
Let $j^*=\pi\circ j$, which is an elementary embedding from $V_\theta$ into $N^*$.
We have that $j^*(\alpha)=\alpha$ for every $\alpha<\theta$.

Now we prove that $j^*(a) =a$ for every $a \in V_\theta$,
however this is impossible;
We know $S \in V_\theta$ but $[\mathrm{id}] \in j^*(S) \setminus S$ since 
$U$ is non-principal.

\begin{claim}
$j^*(a)=a$ for every $a \in V_\theta$.
\end{claim}
\begin{proof}
By induction on the rank of $a$.
Let $\alpha<\theta$,
and suppose 
$j^*(b)=b$ for every $b \in V_\alpha$.
Take a set $a \subseteq V_\alpha$ with rank $\alpha$.
By the induction hypothesis, we have $a \subseteq j^*(a)$.
For the converse that $j^*(a) \subseteq a$,
because $\mathrm{rank}(a)=j^*(\mathrm{rank}(a))=\mathrm{rank}(j^*(a))$,
we have that $j^*(a) \subseteq V_\alpha$.
Pick $b \in j^*(a)$. Then $\mathrm{rank}(b)<\alpha$,
so $j^*(b) =b \in j^*(a)$ and $b \in a$.
\end{proof}
\end{proof}

\begin{lemma}\label{7.5}
Suppose $\LT$ holds but there is no non-principal ultrafilter over $\om$.
Then for every set $S$,
$S$ carries a non-principal ultrafilter if and only if
there is a measurable cardinal $\ka$ and an injection $f:\ka \to S$.
\end{lemma}
\begin{proof}
If there is an injection from a measurable cardinal $\ka$ to $S$,
then it is clear that $S$ carries a non-principal ultrafilter.
For the converse, suppose $S$ carries a non-principal ultrafilter $U$.
By $\LT$ and Proposition \ref{6.32},
there is a cardinal $\ka$ and a map $g:S \to \ka$ such that
$U$ is $\ka$-complete but not $\ka^+$-complete, and $g^{-1}(\alpha) \notin U$ for every $\alpha<\ka$.
By Lemma \ref{7.1} and the assumption that no non-principal ultrafilter over $\om$,
$\ka$ must be uncountable.
$g_*(U)$ forms a non-principal $\ka$-complete ultrafilter, hence
$\ka$ is measurable.
Now we may assume that $g^{-1}(\alpha+1) \setminus g^{-1}(\alpha) \neq \emptyset$ for every $\alpha<\ka$.
Let $S_\alpha=g^{-1}(\alpha+1) \setminus g^{-1}(\alpha)$.
By $\LT$, we can find a function $h$ on $\ka$ such that
$X=\{\alpha<\ka \mid h(\alpha) \in S_\alpha\} \in g_*(U)$.
Then $h \restriction X$ is an injection from $X$ into $S$.
By arranging this function,
we can take an injection from $\ka$ to $S$.
\end{proof}

\begin{prop}\label{7.2}
Suppose $\LT$ holds but there is no non-principal ultrafilter over $\om$.
Then the following are equivalent:
\begin{enumerate}
\item There is a measurable cardinal.
\item There is a non-principal $\sigma$-complete ultrafilter.
\item There is a non-principal ultrafilter.
\item There is a non-principal ultrafilter, and every ultrafilter is $\sigma$-complete.
\end{enumerate}
\end{prop}
\begin{proof}
(1) $\Rightarrow$ (2) $\Rightarrow$ (3) is trivial,
and (3) $\Rightarrow$ (1) is Lemma \ref{7.5}.

(4) $\Rightarrow$ (3) is obvious,
and (3) $\Rightarrow$ (4) follows form
Lemma \ref{7.1}.
\end{proof}

By this proposition, we have Theorem \ref{1.8+}.
\begin{thm}
The following theories are equiconsistent:
\begin{enumerate}
\item $\ZF+\LT+\mathsf{WUF(?)}+\neg \mathsf{AC}_{\om,2}$.
\item $\ZF+\LT+\mathsf{WUF(?)}+\neg \mathsf{AC}_\om$.
\item $\ZFC+$``there is a measurable cardinal''.
\end{enumerate}
\end{thm}
\begin{proof}
(1) $\Rightarrow$ (2) is trivial, and 
(3) $\Rightarrow$ (1) is Theorem \ref{6.13}.

For (2) $\Rightarrow$ (3), suppose $V$ is a model of
$\ZF+\LT+\mathsf{WUF(?)}+\neg \mathsf{AC}_\om$.
Then by Theorem \ref{2.1}, there is no non-principal ultrafilter over $\om$.
Hence there is a measurable cardinal by Proposition \ref{7.2},
and then we can take an inner model of $\ZFC+$``there is a measurable cardinal''.
\end{proof}

We also have:
\begin{cor}\label{7.8}
The following theories are equiconsistent:
\begin{enumerate}
\item $\ZFC+$``there is a measurable cardinal''.
\item $\ZF+\LT+$``there is a non-principal $\sigma$-complete ultrafilter''.
\item $\ZF+\LT+\mathsf{WUF(?)}+$``every ultrafilter is
$\sigma$-complete''.
\end{enumerate}
\end{cor}

Finally, unlike Corollary \ref{7.8}, we prove that the statement that
$\mathsf{WUF(?)}+$``every ultrafilter is $\sigma$-complete'' 
does not have a large cardinal strength.

An infinite set $S$ is said to be \emph{amorphous}
if every subset of $S$ is finite or co-finite.
If $S$ is amorphous, the set $U$ of all co-finite subsets of $S$ 
forms a $\sigma$-complete non-principal ultrafilter over $S$;
It is easy to check that $U$ is a non-principal ultrafilter.
If $U$ is not $\sigma$-complete,
we can find a partition $\seq{S_n \mid n<\om}$ of $S$ such that
each $S_n$ is non-empty finite subset of $S$.
Then the set $\bigcup_{n } S_{2n}$ is an infinite co-infinite subset of $S$,
this is a contradiction.
In fact, one can check that $U$ is $\ka$-complete for every cardinal $\ka$ by
the same argument.

\begin{fact}[E.g., see Jech \cite{J}]
If $\ZFC$ is consistent,
then so is $\ZF+$``there is an amorphous set''.
\end{fact}
\begin{thm}
If $\mathsf{ZFC}$ is consistent, then so is
$\mathsf{ZF}+\mathsf{WUF(?)}+$``every ultrafilter is $\sigma$-complete''.
\end{thm}
\begin{proof}
By Lemma \ref{7.1}, it is enough to construct a model of $\ZF$ in which
there is a non-principal ultrafilter but
there is no non-principal ultrafilter over $\om$.

Since the theory $\ZF$+``there is an amorphous set'' is equiconsistent with
$\ZFC$,
we can assume that there is an amorphous set $S$ in $V$.
Let $U$ be the set of all co-finite subsets of $S$,
it is a non-principal $\sigma$-complete ultrafilter over $S$. 
Let $\bbP=\mathrm{Fn}(\om \times \om,2, \mathop{<}\om)$. Note that $\bbP$ is countable in $V$.
Take a $(V, \bbP)$-generic $G$.

\begin{claim}\label{7.11}
In $V[G]$, $U$ generates an ultrafilter over $S$.
\end{claim}
\begin{proof}
Take $p \in \bbP$ and a $\bbP$-name $\dot {X}$ such that
$p \Vdash \dot X \subseteq \check S$.
For $q \le p$, 
let $S_q=\{s \in S \mid q \Vdash \check s \in \dot X\}$.
If $S_q \in U$ for some $q \in \bbP$, then $q \Vdash \check S_q \subseteq \dot X$,
and $q$ forces that ``$\dot X$ is in the filter generated by $U$''.
If there is no $q\le p$ with $S_q \in U$,
since $\bbP$ is countable and $U$ is $\sigma$-complete,
we have $T=\bigcup_{q \le p} (S \setminus S_q) \in U$
and  $p \Vdash \check T \cap \dot X=\emptyset$,
so $p \Vdash \check S \setminus \dot X \in \dot U$.
\end{proof}

Now, let $\mathrm{HS}^G$ be the symmetric extension of $V$ in Section \ref{sec5} with the case $\ka=\om$.
We can carry out the proof of Lemma \ref{4.5} without $\mathsf{AC}$ in $V$,
hence $\mathrm{HS}^G$ is a model of $\ZF+$``there is no non-principal ultrafilter over $\om$''.
On the other hand, since $\mathrm{HS}^G$ is an intermediate model
between $V$ and $V[G]$,
$U$ generates a non-principal ultrafilter over $S$ in $\mathrm{HS}^G$ by Claim \ref{7.11}.
\end{proof}

\end{document}